\documentclass[a4paper,12pt]{article}
\usepackage{amsmath,amsthm,amssymb,amscd}
\usepackage{enumerate}
\theoremstyle{theorem}
\newtheorem{thm}{Theorem}[section]
\newtheorem{prop}[thm]{Proposition}
\newtheorem{cor}[thm]{Corollary}
\newtheorem{lem}[thm]{Lemma}

\newtheorem{thrm}{Theorem}
\theoremstyle{definition}

\newtheorem{exa}[thm]{Example}
\newtheorem{rem}[thm]{Remark}
\def\rank{\mathop{\mathrm{rank}}\nolimits}
\def\dim{\mathop{\mathrm{dim}}\nolimits}

\def\det{\mathop{\mathrm{det}}\nolimits}

\def\Ker{\mathop{\mathrm{Ker}}\nolimits}

\def\diag{\mathop{\mathrm{diag}}\nolimits}

\def\span{\mathop{\mathrm{span}}\nolimits}
\def\Cyc{\mathop{\mathrm{Cyc}}\nolimits}
\def\triv{\mathop{\mathrm{\mathbf{triv}}}\nolimits}
\def\Res{\mathop{\mathrm{Res}}\nolimits}

\def\Ind{\mathop{\mathrm{Ind}}\nolimits}
\newcommand{\C}{\mathbf{C}}

\newcommand{\Z}{\mathbf{Z}}
\newcommand{\Q}{\mathbf{Q}}
\newcommand{\K}{\mathbf{K}}
\def\triv{\mathop{\mathrm{\mb{triv}}}\nolimits}
\def\Cyc{\mathop{\mathrm{Cyc}}\nolimits}
\def\pr{\mathop{\mathrm{pr}}\nolimits}
\def\LR{\mathop{\mathrm{LR}}\nolimits}
\def\Mob{\mathop{\text{M\"{o}b}}\nolimits}
\def\maj{\mathop{\mathrm{maj}}\nolimits}

\newcommand{\mf}[1]{{\mathfrak{#1}}}
\newcommand{\mb}[1]{{\mathbf{#1}}}

\newcommand{\mca}[1]{{\mathcal{#1}}}
\title{\textbf{On the derivation algebra of the free Lie algebra and trace maps}}
\author{
Naoya Enomoto\footnote{enomoto@math.kyoto-u.ac.jp} \ and \ Takao Satoh\footnote{takao@math.kyoto-u.ac.jp}
}
\date{\empty}
\begin{document}
\maketitle
\begin{abstract}
In this paper, we mainly study the derivation algebra of the free Lie algebra and the Chen Lie algebra generated by the abelianization $H$ of a free group,
and trace maps. 
To begin with, we give the irreducible decomposition of the derivation algebra as a $\mathrm{GL}(n,\Q)$-module
via the Schur-Weyl duality and some tensor product theorem for $\mathrm{GL}(n,\Q)$.
Using them, we calculate the irreducible decomposition of the images of the Johnson homomorphisms
of the automorphism group of a free group and a free metabelian group.

Next, we consider some applications of trace maps: the Morita's trace map and the trace map for the exterior product of $H$.
First, we determine the abelianization of the derivation algebra of the Chen Lie algebra as a Lie algebra, and show that
the abelianizaton is given by the degree one part and the Morita's trace maps.
Second, we consider twisted cohomology groups of the automorphism group of a free nilpotent group.
Especially, we show that the trace map for the exterior product of $H$ defines a non-trivial twisted second cohomology class of it.
\end{abstract}

\section{Introduction}

For a free group $F_n$ with basis $x_1, \ldots, x_n$, set $H:=F_n^{\mathrm{ab}}$ the abelianization of $F_n$.
The kernel of the natural homomorphism $\rho : \mathrm{Aut}\,F_n \rightarrow \mathrm{Aut}\,H$ induced from
the abelianization of $F_n \rightarrow H$ is called the IA-automorphism group of $F_n$, and denoted by $\mathrm{IA}_n$.
Although $\mathrm{IA}_n$ plays important roles on various studies of $\mathrm{Aut}\,F_n$,
the group structure of $\mathrm{IA}_n$ is quite complicated in general. For example,
any presentation for $\mathrm{IA}_n$ is not known. 
Furthermore, Krsti\'{c} and McCool \cite{Krs} showed that $\mathrm{IA}_3$ is not finitely presentable.
For $n \geq 4$, it is not known whether $\mathrm{IA}_n$ is finitely presentable or not.

\vspace{0.5em}

In order to study a deep group structure of $\mathrm{IA}_n$, it is sometimes useful to consider the Johnson filtration.
For the lower central series $\Gamma_n(k)$ of $F_n$, the action of $\mathrm{Aut}\,F_n$ on the nilpotent quotient $F_n/\Gamma_n(k)$ induces
a natural homomorphism $\mathrm{Aut}\,F_n \rightarrow \mathrm{Aut}\,(F_n/\Gamma_n(k+1))$. Then its kernel $\mathcal{A}_n(k)$ defines a descending central filtration
$\mathrm{IA}_n = \mathcal{A}_n(1) \supset \mathcal{A}_n(2) \supset \cdots$. This filtration is called the Johnson filtration of $\mathrm{Aut}\,F_n$.
Each of the graded quotient ${\mathrm{gr}}^k (\mathcal{A}_n):=\mathcal{A}_n(k)/\mathcal{A}_n(k+1)$
naturally has a $\mathrm{GL}(n,\Z)$-module structure, and is considered as one by one approximation of $\mathrm{IA}_n$.
To study ${\mathrm{gr}}^k (\mathcal{A}_n)$,
the Johnson homomorphisms
\[ \tau_k : \mathrm{gr}^k (\mathcal{A}_n) \hookrightarrow H^* \otimes_{\Z} \mathcal{L}_n(k+1) \]
of $\mathrm{Aut}\,F_n$ are defined where $H^* := \mathrm{Hom}_{\Z}(H,\Z)$. 
Historically, the Johnson filtration was originally studied by Andreadakis \cite{And} in 1960's, and
the Johnson homomorphisms by D. Johnson \cite{Jo1} in 1980's who determined the abelianization
of the Torelli subgroup of the mapping class group of a surface in \cite{Jo2}.
Now, there is a broad range of remarkable results
for the Johnson homomorphisms of the mapping class group. (For example, see \cite{Hai}, \cite{Mo1}, \cite{Mo2} and \cite{Mo3}.)
Since each of $\tau_k$ is $\mathrm{GL}(n,\Z)$-equivariant injective,
to clarify the structure of the image of $\tau_k$ is one of the most basic problems.
By a pioneer work of Andreadakis \cite{And}, it is known that $\tau_1$ is an isomorphism.
It is known that $\mathrm{Coker}(\tau_{2,\Q})=S^2 H_{\Q}$ and
$\mathrm{Coker}(\tau_{3,\Q})=S^3 H_{\Q} \oplus \Lambda^3 H_{\Q}$ by Pettet \cite{Pet} and Satoh \cite{S03} respectively.
Here $\tau_{k,\Q}:=\tau_k \otimes \mathrm{id}_{\Q}$ and
$H_{\Q} := H \otimes_{\Z} \Q$.
In general, however, it is quite a difficult problem to determine even the rank of the image of $\tau_k$ for $k \geq 4$.

\vspace{0.5em}

Now, let $\mathcal{A}_n'(1)$, $\mathcal{A}_n'(2)$, $\dots$ be the lower central series of $\mathrm{IA}_n$.
Since the Johnson filtration is central, $\mathcal{A}_n'(k) \subset \mathcal{A}_n(k)$ for each $k \geq 1$.
It is conjectured that $\mathcal{A}_n'(k) = \mathcal{A}_n(k)$ for each $k \geq 1$ by Andreadakis who showed $\mathcal{A}_2'(k) = \mathcal{A}_2(k)$
and $\mathcal{A}_3'(3) = \mathcal{A}_3(3)$.
It is known that $\mathcal{A}_n'(2) = \mathcal{A}_n(2)$ due to
Bachmuth \cite{Ba2}, and that
$\mathcal{A}_n'(3)$ has at most finite index in $\mathcal{A}_n(3)$ due to Pettet \cite{Pet}.
Set ${\mathrm{gr}}^k (\mathcal{A}_n'):=\mathcal{A}_n'(k)/\mathcal{A}_n'(k+1)$.
Then we can also define a $\mathrm{GL}(n,\Z)$-equivariant homomorphism
\[ \tau_k' : \mathrm{gr}^k (\mathcal{A}_n') \rightarrow H^* \otimes_{\Z} \mathcal{L}_n(k+1) \]
by the same way as $\tau_k$. We also call $\tau_k'$ the $k$-th Johnson homomorphism.
In \cite{S11}, we determine the cokernel of the stable rational Johnson homomorphism $\tau_{k,\Q}':=\tau_k' \otimes \mathrm{id}_{\Q}$.
More precisely, we proved that for any $k \geq 2$ and $n \geq k+2$,
\[ \mathrm{Coker}(\tau_{k, \Q}') \cong \mathcal{C}_n^{\Q}(k) \]
where $\mathcal{C}_n^{\Q}(k):=\mathcal{C}_n(k) \otimes_{\Z} \Q$, and
$\mathcal{C}_n(k)$ be a quotient module of $H^{\otimes k}$ by the action of cyclic group $\Cyc_k$ of order $k$ on the components: 
\[ \mathcal{C}_n(k) = H^{\otimes k} \big{/} \langle a_1 \otimes a_2 \otimes \cdots \otimes a_k - a_2 \otimes a_3 \otimes \cdots \otimes a_k \otimes a_1
   \,|\, a_i \in H \rangle. \]

\vspace{0.5em}

In general the target of $\tau_{k}$ is considered as the degree $k$ part
of the derivation algebra $\mathrm{Der}^+(\mathcal{L}_{n})$ of the free Lie algebra $\mathcal{L}_{n}$ generated by $H$.
The first aim of the paper is to give irreducible decompositions of $\mathrm{GL}(n,\Q)$-modules $\mathcal{C}_n^\Q(k)$ and
$\mathrm{Der}^+(\mathcal{L}_{n,\Q})(k)=\mathrm{Der}^+(\mathcal{L}_{n})(k) \otimes_{\Z} \Q$.
Our proof is based on the Schur-Weyl duality for $\mathrm{GL}(n,\Q)$ and $\mathfrak{S}_k$.
\begin{thrm}($=$ Propositions {\rmfamily \ref{prop:1}} and {\rmfamily \ref{prop:3}}, and Corollary {\rmfamily \ref{cor:3}}.)\label{Int_T-1}
\begin{enumerate}
\item[$(1)$]
The multiplicities of irreducible $\mathrm{GL}(n,\Q)$-modules $L^\lambda$ with highest weight $\lambda$ in $\mathcal{C}_n^\Q(k)$ are given by
\[ [L^\lambda:\mathcal{C}_n^{\Q}(k)] =  \begin{cases}
                           [\triv_k:\Res_{\Cyc_k}^{\mf{S}_k}S^\lambda]
                             & \mathrm{if} \,\,\, \text{$\lambda$ is a partition of $k$}, \\
                           0, & \mathrm{if} \,\,\, \mathrm{otherwise},
                         \end{cases} \]
where $\mf{S}_k$ is the symmetric group of degree $k$, $S^\lambda$ is its irreducible module associated to a partition $\lambda$ of $k$, $\Cyc_k$
is a cyclic subgroup of $\mf{S}_k$ generated by a cyclic permutation of order $k$ and $\triv_k$ is the trivial representation of $\Cyc_k$.

\item[$(2)$] For any $k \geq 1$ and $n \geq k+2$, as a $\mathrm{GL}(n, \Q)$-module, we have a direct decomposition 
\[ \mathrm{Der}^+(\mathcal{L}_{n,\Q})(k) = (H_{\Q})^{\otimes{k}}
   \oplus \bigoplus_{\mu; \,\, \ell(\mu) \le n}[L^\mu: \mathcal{L}_{n}(k)] L^{\{\mu;(1)\}}, \]
where in the second term, the sum runs over all partitions $\mu$ such that its length $\ell(\mu)$ is smaller than or equal to $n$,
and $L^{\{\mu;(1)\}}$ is the irreducible $\mathrm{GL}(n,\Q)$-module $\det^{-1} \otimes L^\mu$.
\end{enumerate}
\end{thrm}

We remark that, as a $\mathrm{GL}(n,\Q)$-module, $\mathcal{C}_n^{\Q}(k)$ is isomorphic to the invariant part $a_n(k):=(H_{\Q}^{\otimes k})^{\Cyc_k}$ of $H_{\Q}^{\otimes k}$
by the action of $\Cyc_k$. Namely, the cokernel $\mathrm{Coker}(\tau_{k, \Q}')$ is isomorphic to Kontsevich's
$a_n(k)$ as a $\mathrm{GL}(n,\Q)$-module. In our notation $a_n(k)$ is considered for any $n \geq 2$ in constrast to Kontsevich's notation
for even $n=2g$. (See \cite{Kn1} and \cite{Kn2}.)
We also remark that the polynomial part $(H_{\Q})^{\otimes{k}}$ of $\mathrm{Der}^+(\mathcal{L}_{n,\Q})(k)$ is detected by a contraction map.
(See Subsection {\rmfamily \ref{Ss-Der}} for the definition of the contraction map.)
Using this theorem, for given $k \geq 1$, we can calculate the irreducible decomposition of $\mathrm{Im}(\tau_{n,\Q}')$ and $\mathrm{Coker}(\tau_{n,\Q}')$
for $n \geq k+2$.

\vspace{0.5em}

On the other hand, we also give the irreducible decompositions of the derivation algebras
of the Chen Lie algebra $\mathcal{L}_{n,\Q}^M$ and some free abelian by polynilpotent Lie algebra $\mathcal{L}_{n,\Q}^N$ generated by $H_{\Q}$.
(See Subsections {\rmfamily \ref{Ss-Chen}} and {\rmfamily \ref{Ss-satoh}} for the precise definition.)
They were studied in our previous papers \cite{S06} and \cite{S08} in order to investigate the cokernel of the Johnson homomorphisms $\tau_{k,\Q}'$.
In this paper, after tensoring with $\Q$, we determine the irreducible decompositions of the image of the Johnson homomorphism $\tau_k^M$
of the automorphism group of a free metabelian group, and of the image of the composition map $\tau_{k, N}'$ of $\tau_k'$ and a homomorphism
$H^* \otimes_{\Z} \mathcal{L}_n(k+1) \rightarrow H^* \otimes_{\Z} \mathcal{L}_n^N(k+1)$
induced from the natural projection $\mathcal{L}_n(k+1) \rightarrow \mathcal{L}_n^N(k+1)$. That is,
\begin{thrm}($=$ Propositions {\rmfamily \ref{prop:John^M}} and {\rmfamily \ref{prop:John^N}}.) 
\begin{enumerate}
\item[$(1)$] For any $k \geq 1$ and $n \geq k+2$,
\[ \mathrm{Im}(\tau_{k,\Q}^M) \cong L^{\{(k,1),(1)\}} \oplus L^{\{(k-1,1),0\}}. \]
\item[$(2)$] For any $k \geq 1$ and $n \geq k+2$,
{\small
\begin{eqnarray*}
  \mathrm{Im}((\tau_{k, N}')_{\Q})
    &\cong& 3L^{(k-1,1)} \oplus 2L^{(k-2,2)} \oplus 2L^{(k-2,1^2)} \oplus L^{(k-3,2,1)} \oplus L^{(k-3,1^3)} \\
    & & \oplus L^{\{(k,1);(1)\}} \oplus L^{\{(k-1,2);(1)\}}
   \oplus L^{\{(k-1,1^2);(1)\}} \oplus L^{\{(k-2,2,1);(1)\}} \oplus L^{\{(k-2,1^3);(1)\}}
\end{eqnarray*}}
\end{enumerate}
where $\tau_{k,\Q}^M:=\tau_{k}^M \otimes \mathrm{id}_{\Q}$ and $(\tau_{k, N}')_{\Q}:=\tau_{k, N}' \otimes \mathrm{id}_{\Q}$.
\end{thrm}
For any irreducible $\mathrm{GL}(n,\Q)$-module $L^\lambda$ with highest weight $\lambda$, we see
\[ [L^\lambda : \mathrm{Im}(\tau_{k,\Q}^M)] \leq [L^\lambda : \mathrm{Im}((\tau_{k, N}')_{\Q})]
    \leq [L^\lambda : \mathrm{Im}(\tau_{k,\Q}')] \leq [L^\lambda : \mathrm{Im}(\tau_{k,\Q})]. \]
Hence we can regard $\mathrm{Im}(\tau_{k,\Q}^M)$ and $\mathrm{Im}((\tau_{k, N}')_{\Q})$ as lower bounds on $\mathrm{Im}(\tau_{k,\Q}')$ and
$\mathrm{Im}(\tau_{k,\Q})$.

\vspace{0.5em}

In the rest of the paper, we consider some applications of trace maps.
In general, the trace maps are used to study the cokernel of the Johnson homomorphisms. One of the most important trace maps is Morita's trace map
\[ \mathrm{Tr}_{[k]} := f_{[k]} \circ \Phi_{1}^k : H^* {\otimes}_{\Z} \mathcal{L}_n(k+1) \rightarrow S^k H. \]
Introducing this map, Morita showed that $S^k H_{\Q}$ appears in the irreducible decomposition of $\mathrm{Coker}(\tau_{k,\Q})$.
Using the part (1) of Theorem {\rmfamily \ref{Int_T-1}},
we see that the multiplicity of $S^k H_{\Q}$ in the irreducible decomposition
of each of $\mathrm{Coker}(\tau_{n,\Q}')$ and $\mathrm{Coker}(\tau_{n,\Q})$ is just one.
In Section {\rmfamily \ref{S-Chen}}, we determine the abelianization of the derivation algebra
$\mathrm{Der}^+(\mathcal{L}_n^M)$ of the Chen Lie algebra using Morita's trace maps. That is,
\begin{thrm}($=$ Theorem {\rmfamily \ref{T-ES_Chen}}.)
For $n \geq 4$, we have
\[ (\mathrm{Der}^+(\mathcal{L}_n^M))^{\mathrm{ab}} \cong (H^* \otimes_{\Z} \Lambda^2 H ) \oplus \bigoplus_{k \geq 2} S^k H. \]
More precisely, this isomorphism is given by the degree one part and Morita's trace maps $\mathrm{Tr}_{[k]}$.
\end{thrm}
We should remark that this result is the Chen Lie algebra version of Morita's conjecture:
\[ (\mathrm{Der}^+(\mathcal{L}_n))^{\mathrm{ab}} \cong (H^* \otimes_{\Z} \Lambda^2 H ) \oplus \bigoplus_{k \geq 2} S^k H \]
for the free Lie algebra $\mathcal{L}_n$ for any $n \geq 3$. (See also \cite{Mo3}.)

\vspace{0.5em}

Next we consider another important trace map
\[ \mathrm{Tr}_{[1^k]} := f_{[1^k]} \circ \Phi_{1}^k : H^* {\otimes}_{\Z} \mathcal{L}_n(k+1) \rightarrow {\Lambda}^k H, \]
called the trace map for the exterior product $\Lambda^k H$.
In \cite{S03}, we show that $\Lambda^k H_{\Q}$ appears in $\mathrm{Coker}(\tau_{k,\Q}')$ for odd $k$ and $3 \leq k \leq n$, and
determine $\mathrm{Coker}(\tau_{3,\Q})$ using $\mathrm{Tr}_{[3]}$ and $\mathrm{Tr}_{[1^3]}$.
In Section {\rmfamily \ref{S-Coh}}, we prove that the trace map $\mathrm{Tr}_{[1^k]}$ defines a non-trivial twisted second cohomology class
of the automorphism group $\mathrm{Aut}\,N_{n,k}$ of a free nilpotent group $N_{n,k}:=F_n/\Gamma_n(k+1)$ with coefficients in
$\Lambda^k H_{\Q}$ for any $k \geq 2$ and $n \geq k$.
To show this, we prove that $H^1(\mathrm{Aut}\,N_{n,k}, \Lambda^k H_{\Q})$ is trivial. Then we consider the cohomological five term
exact sequence of a group extension
\[ 0 \rightarrow \mathrm{Hom}_{\Z}(H, \mathcal{L}_n(k+1)) \rightarrow \mathrm{Aut}\, N_{n,k+1} \rightarrow \mathrm{Aut}\,N_{n,k} \rightarrow 1 \]
introduced by Andreadakis \cite{And}. (See also Proposition 2.3 in Morita's paper \cite{Mo4}.)

\vspace{0.5em}

On the other hand, let $T_{n,k}$ be the image of a natural homomorphism $\mathrm{Aut}\,F_n \rightarrow \mathrm{Aut}\,N_{n,k}$ induced from
the projection $F_n \rightarrow N_{n,k}$. The group $T_{n,k}$ is called the tame automorphism group of $N_{n,k}$. Similarly to $\mathrm{Aut}\,N_{n,k}$,
observing the cohomological five term exact sequence of a group extension
\[ 0 \rightarrow \mathrm{gr}^k(\mathcal{A}_n) \rightarrow T_{n,k+1} \rightarrow T_{n,k} \rightarrow 1, \]
we show that the $\mathrm{GL}(n,\Z)$-equivariant homomorphism $\mathrm{Tr}_{[1^k]} \circ \tau_k$ defines a non-trivial twisted second cohomology class
of $T_{n,k}$.
Namely, the main purpose in Section {\rmfamily \ref{S-Coh}} is to show
\begin{thrm}($=$ Propositions {\rmfamily \ref{T-ES_TAut}} and {\rmfamily \ref{T-ES_NAut}}.)
\begin{enumerate}
\item[$(1)$] $0 \neq \mathrm{tg}(\mathrm{Tr}_{[1^k]} \circ \tau_{k}) \in  H^2(T_{n,k}, \Lambda^k H_{\Q})$
for even $k$ and $2 \leq k \leq n$,
\item[$(2)$] $0 \neq \mathrm{tg}(\mathrm{Tr}_{[1^k]}) \in H^2(\mathrm{Aut}\,N_{n,k}, \Lambda^k H_{\Q})$ for $k \geq 3$ and $n \geq k$
\end{enumerate}
where $\mathrm{tg}$ means the transgression map.
\end{thrm}
In Section {\rmfamily \ref{S-Coh}},
We also show that $H^1(T_{n,k}, H)=\Z$ and it is generated by the Morita's crossed homomorphism. (See Proposition {\rmfamily \ref{P-TAut}}.)

\tableofcontents

\section{Preliminaries}\label{S-Pre}

In this section, after fixing notation and conventions,
we briefly recall some facts of the automorphism group of a free group, the free Lie algebra, the Chen Lie algebra
and the automorphism group of a free nilpotent group.

\subsection{Notation and conventions}\label{Ss-Not}
\hspace*{\fill}\ 

Throughout the paper, we use the following notation and conventions. Let $G$ be a group and $N$ a normal subgroup of $G$.
\begin{itemize}
\item The abelianization of $G$ is denoted by $G^{\mathrm{ab}}$. Namely, $G^{\mathrm{ab}}=H_1(G,\Z)$. Similarly, for any Lie algebra $\mathcal{G}$,
      we denote by $\mathcal{G}^{\mathrm{ab}}$ the abelianization of $\mathcal{G}$ as a Lie algebra.
\item The automorphism group $\mathrm{Aut}\,G$ of $G$ acts on $G$ from the right unless otherwise noted.
      For any $\sigma \in \mathrm{Aut}\,G$ and $x \in G$, the action of $\sigma$ on $x$ is denoted by $x^{\sigma}$.
\item For an element $g \in G$, we also denote the coset class of $g$ by $g \in G/N$ if there is no confusion.
\item For elements $x$ and $y$ of $G$, the commutator bracket $[x,y]$ of $x$ and $y$
      is defined to be $[x,y]:=xyx^{-1}y^{-1}$.
\item For elements $g_1, \ldots, g_k \in G$, a commutator of weight $k$ of the type
\[ [[ \cdots [[ g_{1},g_{2}],g_{3}], \cdots ], g_{k}] \]
with all of its brackets to the left of all the elements occurring is called a simple $k$-fold
commutator, and is denoted by $[g_{i_1},g_{i_2}, \cdots, g_{i_k}]$.
\item For any $\Z$-module $M$ and a commutative ring $R$, we denote $M \otimes_{\Z} R$ by the symbol obtained by attaching a subscript $R$ to $M$, like
      $M_{R}$ or $M^{R}$. Similarly, for any $\Z$-linear map $f: A \rightarrow B$,
      the induced $R$-linear map $A_{R} \rightarrow B_{R}$ is denoted by $f_{R}$
      or $f^{R}$.
\end{itemize}

\subsection{Automorphism group of a free group and its subgroups}\label{Ss-Aut}
\hspace*{\fill}\ 

Here we review some properties of the automorphism group of a free group.
To begin with, we recall the Nielsen's finite presentation for $\mathrm{Aut}\,F_n$.
In this paper, we fix a basis $x_1, \ldots , x_n$ of a free group $F_n$ of rank $n$.
Let $P$, $Q$, $S$ and $U$ be automorphisms of $F_n$ defined as follows:
\begin{center}
\begin{tabular}{c|c|c|c|c|c|c} \hline
           & $x_1$      & $x_2$ & $x_3$ & $\cdots$ & $x_{n-1}$ & $x_n$ \\ \hline
  $P$      & $x_2$      & $x_1$ & $x_3$ & $\cdots$ & $x_{n-1}$ & $x_n$ \\ 
  $Q$      & $x_2$      & $x_3$ & $x_4$ & $\cdots$ & $x_{n}$   & $x_1$ \\ 
  $S$      & $x_1^{-1}$ & $x_2$ & $x_3$ & $\cdots$ & $x_{n-1}$ & $x_n$ \\ 
  $U$      & $x_1 x_2$  & $x_2$ & $x_3$ & $\cdots$ & $x_{n-1}$ & $x_n$ \\ \hline
\end{tabular}
\end{center}
Namely, $P$ is an automorphism induced from the permutation of $x_1$ and $x_2$, $Q$ is induced from the cyclic permutation of the basis, and so on.
Nielsen \cite{Ni1} obtained the first finite presentation for $\mathrm{Aut}\,F_n$ in 1924.
\begin{thm}[Nielsen \cite{Ni1}]
For $n \geq 3$, $\mathrm{Aut}\,F_{n}$ has a finite presentation with generators $P$, $Q$, $S$ and $U$
subject to relations: \\
\hspace{2em} {\bf{(R1)}}: $P^2=1$, \\
\hspace{2em} {\bf{(R2)}}: $(QP)^{n-1}=Q^n=1$, \\
\hspace{2em} {\bf{(R3)}}: $[P, Q^{-i} P Q^i]=1$, \hspace{1em} $(2 \leq i \leq [n/2])$, \\
\hspace{2em} {\bf{(R4)}}: $S^2=1$, \\
\hspace{2em} {\bf{(R5)}}: $[S, Q^{-1} P Q] = [S, QP] = 1$, \\
\hspace{2em} {\bf{(R6)}}: $(P S)^4 = (PSPU)^2=1$, \\
\hspace{2em} {\bf{(R7)}}: $[U, Q^{-2}PQ^2] = [U, Q^{-2}UQ^2]= 1$, \hspace{0.5em} $(n \geq 4)$, \\
\hspace{2em} {\bf{(R8)}}: $[U, Q^{-2} S Q^2] = [U, S U S] = 1$, \\
\hspace{2em} {\bf{(R9)}}: $[U, QPQ^{-1}PQ] = [U, PQ^{-1}S U S QP]=1$, \\
\hspace{2em} {\bf{(R10)}}: $[U, PQ^{-1}PQPUPQ^{-1}PQP] =1$, \\
\hspace{2em} {\bf{(R11)}}: $PUPSU =USPS$, \\
\hspace{2em} {\bf{(R12)}}: $(PQ^{-1}UQ)^2 UQ^{-1} U^{-1}QU^{-1} = 1$. 
\end{thm}

In Section {\rmfamily \ref{S-Coh}}, we use the Nielsen's presentation to compute twisted first cohomology groups of $\mathrm{Aut}\,F_{n}$.
Let $X^{\pm1} := \{ x_1^{\pm1}, x_2^{\pm1}, \ldots , x_n^{\pm1} \} \subset F_n$ be a subset of all letters of $F_n$.
We denote by $\Omega_n$ a subgroup of $\mathrm{Aut}\,F_{n}$ consisting of all $\sigma \in \mathrm{Aut}\,F_n$ that effect
a permutation on $X^{\pm1}$. Then it is known that $\Omega_n$ is a finite group of order $2^n n!$ and generated by $P$, $Q$ and $S$. (See \cite{Ni1}.)
The subgroup $\Omega_n$ is called the extended symmetric group of degree $n$.

\vspace{0.5em}

Next we consider the natural projection induced from the abelianization of $F_n$.
Let $H:=F_n^{\mathrm{ab}}$ be the abelianization of $F_n$ and $\rho : \mathrm{Aut}\,F_n \rightarrow \mathrm{Aut}\,H$ the natural homomorphism induced from
the abelianization of $F_n \rightarrow H$. Throughout the paper, we identify $\mathrm{Aut}\,H$ with the general linear group $\mathrm{GL}(n,\Z)$ by
fixing a basis of $H$ induced from the basis $x_1, \ldots , x_n$ of $F_n$.
Using the Nielsen's presentation, we easily see that $\rho$ is surjective.
For any $\sigma \in \mathrm{Aut}\,F_n$, we also denote $\rho(\sigma) \in \mathrm{GL}(n,\Z)$ by $\sigma$
if there is no confusion. With this notation, $P$, $Q$, $S$ and $U$ generate $\mathrm{GL}(n,\Z)$. In particular, it is known that
\begin{thm}[Magnus \cite{Mag}]\label{T-Gen} (See also Section 7.3 in \cite{CoM}.)
For $n \geq 3$, the group $\mathrm{GL}(n,\Z)$ has a finite presentation with generators $P$, $Q$, $S$ and $U$
subject to relations {\bf{(R1)}}, $\ldots$, {\bf{(R12)}} and \\
\hspace{2em} {\bf{(R13)}}: $(SU)^2=1$. 
\end{thm}

\vspace{0.3em}

Now the kernel $\mathrm{IA}_n$ of $\rho$ is called the IA-automorphism group of $F_n$.
Magnus \cite{Mag} showed that for any $n \geq 3$, $\mathrm{IA}_n$ is finitely generated by automorphisms
\[ K_{ij} : x_t \mapsto \begin{cases}
               {x_j}^{-1} x_i x_j, & t=i, \\
               x_t,                & t \neq i
              \end{cases}\]
for distinct $1 \leq i, \, j \leq n$, and
\[  K_{ijl} : x_t \mapsto \begin{cases}
               x_i [x_j, x_l], & t=i, \\
               x_t,            & t \neq i
              \end{cases}\] 
for distinct $1 \leq i, \, j, \, l \leq n$ and $j<l$.
Recently, Cohen-Pakianathan \cite{Co1, Co2}，Farb \cite{Far} and Kawazumi \cite{Kaw} independently showed
\begin{equation}\label{CPFK}
\mathrm{IA}_n^{\mathrm{ab}} \cong H^* \otimes_{\Z} \Lambda^2 H
\end{equation}
as a $\mathrm{GL}(n,\Z)$-module where $H^*:= \mathrm{Hom}_{\Z}(H,\Z)$ is the $\Z$-linear dual group of $H$.
In particular, from their result, we see that $\mathrm{IA}_n^{\mathrm{ab}}$ is a free abelian group of rank $2n^2(n-1)$ with basis
the coset classes of the Magnus generators $K_{ij}$ and $K_{ijl}$.

\vspace{0.5em}

We denote by $\overline{\Omega}_n$ the image of $\Omega_n$ by the natural projection $\rho$. Since $\mathrm{IA}_n$ is torsion free,
$\overline{\Omega}_n$ is isomorphic to $\Omega_n$. Namely, $\overline{\Omega}_n$ is a finite group of order $2^n n!$ generated by $P$, $Q$ and $S$.

\subsection{Free Lie algebra $\mathcal{L}_n$ and its derivations}\label{Ss-Der}
\hspace*{\fill}\ 

In this subsection, we recall the free Lie algebra generated by $H$, and its derivation algebra.
Let $\Gamma_n(1) \supset \Gamma_n(2) \supset \cdots$ be the lower central series of a free group $F_n$ defined by the rule
\[ \Gamma_n(1):= F_n, \hspace{1em} \Gamma_n(k) := [\Gamma_n(k-1),F_n], \hspace{1em} k \geq 2. \]
We denote by $\mathcal{L}_n(k) := \Gamma_n(k)/\Gamma_n(k+1)$ the $k$-th graded quotient of the lower central series of $F_n$,
and by $\mathcal{L}_n := {\bigoplus}_{k \geq 1} \mathcal{L}_n(k)$ the associated graded sum.
It is classically well known due to Witt \cite{Wit} that each $\mathcal{L}_n(k)$ is a free abelian group of rank
\begin{equation}\label{ex-witt}
 \mathrm{rank}_{\Z}(\mathcal{L}_n(k))=\frac{1}{k} \sum_{d | k} \Mob(d) n^{\frac{k}{d}}
\end{equation}
where $\Mob$ is the M$\ddot{\mathrm{o}}$bius function.
The graded sum $\mathcal{L}_n$ naturally has a graded Lie algebra structure induced from
the commutator bracket on $F_n$, and called the free Lie algebra generated by $H$.
(See \cite{Reu} for basic material concerning the free Lie algebra.)
For each $k \geq 1$, $\mathrm{Aut}\,F_n$ naturally acts on $\mathcal{L}_n(k)$.
Since the action of $\mathrm{IA}_n$ on $\mathcal{L}_n(k)$ is trivial,
that of $\mathrm{GL}(n,\Z) = \mathrm{Aut}\,F_n/\mathrm{IA}_n$ on $\mathcal{L}_n(k)$ is well-defined.

\vspace{0.5em}

Next, we consider an embedding of the free Lie algebra into the tensor algebra.
Let
\[ T(H):= \Z \oplus H \oplus H^{\otimes 2} \oplus \cdots \]
be the tensor algebra of $H$ over $\Z$. Then $T(H)$ is the
universal enveloping algebra of the free Lie algebra $\mathcal{L}_n$, and the natural map
$\iota : \mathcal{L}_n \rightarrow T(H)$ defined by
\[ [X,Y] \mapsto X \otimes Y - Y \otimes X \]
for $X$, $Y \in \mathcal{L}_n$ 
is an injective graded Lie algebra homomorphism.
We denote by $\iota_k$ the homomorphism of degree $k$ part of $\iota$, and
consider $\mathcal{L}_n(k)$ as a submodule $H^{\otimes k}$ through $\iota_k$.

\vspace{0.5em}

Here, we recall the derivation algebra of the free Lie algebra.
Let $\mathrm{Der}(\mathcal{L}_n)$ be the graded Lie algebra of derivations of $\mathcal{L}_n$.
Namely,
\[ \mathrm{Der}(\mathcal{L}_n) := \{ f : \mathcal{L}_n \xrightarrow{\Z-\mathrm{linear}} \mathcal{L}_n \,|\, f([a,b]) = [f(a),b]+ [a,f(b)], \,\,\,
   a, b \in \mathcal{L}_n \}. \]
For $k \geq 0$, the degree $k$ part of $\mathrm{Der}(\mathcal{L}_n)$ is defined to be
\[ \mathrm{Der}(\mathcal{L}_n)(k) := \{ f \in \mathrm{Der}(\mathcal{L}_n) \,|\, f(a) \in \mathcal{L}_n(k+1), \,\,\, a \in H \}. \]
Then, we have
\[ \mathrm{Der}(\mathcal{L}_n) = \bigoplus_{k \geq 0} \mathrm{Der}(\mathcal{L}_n)(k), \]
and can consider $\mathrm{Der}(\mathcal{L}_n)(k)$ as
\[ \mathrm{Hom}_{\Z}(H,\mathcal{L}_n(k+1)) = H^* {\otimes}_{\Z} \mathcal{L}_n(k+1) \]
for each $k \geq 1$ by the universality of the free Lie algebra.
Let $\mathrm{Der}^+(\mathcal{L}_n)$ be a graded Lie subalgebra of $\mathrm{Der}(\mathcal{L}_n)(k)$ with positive degree.
(See Section 8 of Chapter II in \cite{Bou}.)
Similarly, we define a graded Lie algebra $\mathrm{Der}^+(\mathcal{L}_{n,\Q})$ over $\Q$.
Then, we have $\mathrm{Der}^+(\mathcal{L}_{n,\Q}) = \mathrm{Der}^+(\mathcal{L}_n) \otimes_{\Z} \Q$.

\vspace{0.5em}

For $k \geq 1$,
let ${\varphi}^{k} : H^* {\otimes}_{\Z} H^{\otimes (k+1)} \rightarrow H^{\otimes k}$
be the contraction map defined by
\[ x_i^* \otimes x_{j_1} \otimes \cdots \otimes x_{j_{k+1}} \mapsto x_i^*(x_{j_1}) \, \cdot
    x_{j_2} \otimes \cdots \otimes \cdots \otimes x_{j_{k+1}}. \]
For the natural embedding ${\iota}_{k+1} : \mathcal{L}_n(k+1) \rightarrow H^{\otimes (k+1)}$,
we obtain a $\mathrm{GL}(n,\Z)$-equivariant homomorphism
\[ \Phi^k = {\varphi}^{k} \circ ({id}_{H^*} \otimes {\iota}_{k+1})
    : H^* {\otimes}_{\Z} \mathcal{L}_n(k+1) \rightarrow H^{\otimes k}. \]
We also call $\Phi^k$ a contraction map.
In Proposition {\rmfamily \ref{prop:3}}, we study the irreducible decomposition of $\mathrm{Der}^+(\mathcal{L}_{n,\Q})(k)$
as a $\mathrm{GL}(n,\Z)$-module using the contraction map.

\vspace{0.5em}

Finally, we review the trace maps. For any $k \geq 2$, let $f_{[k]} : H^{\otimes k} \rightarrow S^k H$ be the natural projection defined by
\[ x_{i_1} \otimes \cdots \otimes x_{i_k} \mapsto x_{i_1} \cdots x_{i_k}. \]
Then the composition map
\[ \mathrm{Tr}_{[k]} := f_{[k]} \circ \Phi^k : H^* {\otimes}_{\Z} \mathcal{L}_n(k+1) \rightarrow S^k H \]
is a $\mathrm{GL}(n,\Z)$-equivariant surjective homomorphism, and is called Morita's trace map.
Morita studied the cokernel of the Johnson homomorphism of $\mathrm{Aut}\,F_n$ using Morita's trace map, and showed that
$S^k H_{\Q}$ appears in the irreducible decomposition of it for any $k \geq 2$. 

\vspace{0.5em}

Let $f_{[1^k]} : H^{\otimes k} \rightarrow \Lambda^k H$
be the natural projection defined by
\[ x_{i_1} \otimes \cdots \otimes x_{i_k} \mapsto x_{i_1} \wedge \cdots \wedge x_{i_k}. \]
Then the composition map
\[ \mathrm{Tr}_{[1^k]} := f_{[1^k]} \circ \Phi^k : H^* {\otimes}_{\Z} \mathcal{L}_n(k+1) \rightarrow \Lambda^k H \]
is called the trace map for $\Lambda^k H$. Using the trace map $\mathrm{Tr}_{[1^k]}$, we showed that there appears $\Lambda^k H_{\Q}$
in the cokernel of the Johnson homomorphism restricted to the lower central series of $\mathrm{IA}_n$ if $3 \leq k \leq n$ and $k$ is odd.
(See \cite{S03} for details.) In Section {\rmfamily \ref{S-Coh}}, we show that $\mathrm{Tr}_{[1^k]}$ defines a non-trivial twisted second cohomology class
of the automorphism group of a free nilpotent group.

\subsection{Chen Lie algebra $\mathcal{L}_n^M$ and its derivations}\label{Ss-Chen}
\hspace*{\fill}\ 

Here we recall the Chen Lie algebra generated by $H$, and its derivation algebra.
Let $F_n^M:=F_n/[[F_n,F_n],[F_n,F_n]]$ be a free metabelian group of rank $n$.
Let $\Gamma_n^M(1) \supset \Gamma_n^M(2) \supset \cdots$ be the lower central series of a free group $F_n^M$ defined by the rule
\[ \Gamma_n^M(1):= F_n^M, \hspace{1em} \Gamma_n^M(k) := [\Gamma_n^M(k-1),F_n^M], \hspace{1em} k \geq 2. \]
We denote by $\mathcal{L}_n^M(k) := \Gamma_n^M(k)/\Gamma_n^M(k+1)$ the $k$-th graded quotient of the lower central series of $F_n^M$,
and by $\mathcal{L}_n^M := {\bigoplus}_{k \geq 1} \mathcal{L}_n^M(k)$ the associated graded sum.
The graded sum $\mathcal{L}_n^M$ naturally has a graded Lie algebra structure induced from
the commutator bracket on $F_n^M$ by the same argument as the free Lie algebra $\mathcal{L}_n$.
The Lie algebra $\mathcal{L}_n^M$ is called the free metabelian Lie algebra or the Chen Lie algebra, generated by $H$.

\vspace{0.5em}

Since $(F_n^M)^{\mathrm{ab}}=H$, $\mathrm{Aut}\,(F_n^M)^{\mathrm{ab}}=\mathrm{Aut}(H)=\mathrm{GL}(n,\Z)$.
By an argument similar to the free Lie algebra, it turns out that each of the graded quotients $\mathcal{L}_n^M(k)$ is a $\mathrm{GL}(n,\Z)$-module.
For $1 \leq k \leq 3$, we have $\mathcal{L}_n(k) = \mathcal{L}_n^M(k)$.
It is also classically known due to Chen \cite{Che} that each $\mathcal{L}_n^M(k)$ is a free abelian group of rank
\begin{equation}\label{ex-chen}
 \mathrm{rank}_{\Z}(\mathcal{L}_n^M(k))=(k-1) \binom{n+k-2}{k}
\end{equation}
with basis
\begin{equation}\label{basis-chen}
\{ [x_{i_1}, x_{i_2}, \ldots , x_{i_k} ] \,|\, i_1 > i_2 \leq i_3 \leq  \cdots \leq i_k \}.
\end{equation}

\vspace{0.5em}

Let $\mathrm{Der}^+(\mathcal{L}_n^M)$ be the graded Lie algebra of derivations of $\mathcal{L}_n^M$ with positive degree.
The degree $k$ part of $\mathrm{Der}^+(\mathcal{L}_n^M)$ is considered as
\[ \mathrm{Der}^+(\mathcal{L}_n^M)(k) = H^* {\otimes}_{\Z} \mathcal{L}_n^M(k+1). \]
Similarly, we define a graded Lie algebra $\mathrm{Der}^+(\mathcal{L}_{n,\Q}^{M})$ over $\Q$.
In Subsection {\rmfamily \ref{Ss-DecM}}, we give the irreducible decomposition of $\mathrm{Der}^+(\mathcal{L}_{n,\Q}^{M})$.

\vspace{0.5em}

Now, Morita's trace map $\mathrm{Tr}_{[k]}$ naturally factors through a surjective homomorphism
$H^* \otimes_{\Z} \mathcal{L}_n(k) \rightarrow H^* \otimes_{\Z} \mathcal{L}_n^M(k)$. Namely, $\mathrm{Tr}_{[k]}$ induces a $\mathrm{GL}(n,\Z)$-equivariant
homomorphism
\[ \mathrm{Tr}_{[k]}^M : H^* {\otimes}_{\Z} \mathcal{L}_n(k+1) \rightarrow S^k H. \]
(See Subsection 3.2 in\cite{S06}.) We also call it Morita's trace map.
In Section {\rmfamily \ref{S-Chen}}, we determine the abelianization of $\mathrm{Der}^+(\mathcal{L}_n^M)$
as a Lie algebra using Morita's trace maps $\mathrm{Tr}_{[k]}^M$.

\subsection{Lie algebra $\mathcal{L}_n^N$ and its derivations}\label{Ss-satoh}
\hspace*{\fill}\ 

Let $F_n^N$ be the quotient group of $F_n$ by the subgroup
$[\Gamma_n(3), \Gamma_n(3)][[\Gamma_n(2),\Gamma_n(2)],\Gamma_n(2)]$.
Let $\Gamma_n^N(1) \supset \Gamma_n^N(2) \supset \cdots$ be the lower central series of a free group $F_n^N$ and
$\mathcal{L}_n^N(k):= \Gamma_n^N(k)/\Gamma_n^N(k+1)$ its graded quotients. Similarly, the graded sum
$\mathcal{L}_n^N := {\bigoplus}_{k \geq 1} \mathcal{L}_n^N(k)$ has a graded Lie algebra structure induced from the commutator bracket of $F_n^N$.

\vspace{0.5em}

In our paper \cite{S08}, we have determined the $\Z$-module structure of $\mathcal{L}_n^N(k)$, and showed that
its rank is given by.
\begin{equation}\label{ex-satoh}
 \mathrm{rank}_{\Z}(\mathcal{L}_n^N(k)) = (k-1)\binom{k+n-2}{k} + \frac{1}{2} n(n-1)(k-3) \binom{n+k-4}{k-2},
\end{equation}
and
\begin{eqnarray}
  & & \{ [x_{i_1}, x_{i_2}, \ldots , x_{i_k} ] \,|\, i_1 > i_2 \leq i_3 \leq  \cdots \leq i_k \} \nonumber \\
  & & \hspace{3em} \cup \,\, \{ [x_{i_1}, \ldots, x_{i_{k-2}},[x_{i_{k-1}},x_{i_k}]] \,|\, i_1 > i_2 \leq i_3 \leq \cdots \leq i_{k-2}, \,\, i_{k-1} > i_k \}.\label{basis-N}
\end{eqnarray}
is a basis of it.

\vspace{0.5em}

We \cite{S08} used these facts to investigate the cokernel of the Johnson homomorphism.
Let $\mathrm{Der}^+(\mathcal{L}_n^N)$ be the graded Lie algebra of derivations of $\mathcal{L}_n^N$ with positive degree.
The degree $k$ part of $\mathrm{Der}^+(\mathcal{L}_n^N)$ is considered as
\[ \mathrm{Der}^+(\mathcal{L}_n^N)(k) = H^* {\otimes}_{\Z} \mathcal{L}_n^N(k+1). \]
Similarly, we define a graded Lie algebra $\mathrm{Der}^+(\mathcal{L}_{n,\Q}^{N})$ over $\Q$.
In Subsection {\rmfamily \ref{Ss-DecN}}, we give the irreducible decomposition of $\mathrm{Der}^+(\mathcal{L}_{n,\Q}^{N})$.

\subsection{Johnson homomorphisms}\label{Ss-John}
\hspace*{\fill}\ 

For each $k \geq 1$, let $N_{n,k}:=F_n/\Gamma_n(k+1)$ of $F_n$ be the free nilpotent group of class $k$ and rank $n$, and $\mathrm{Aut}\,N_{n,k}$
its automorphism group.
Since the subgroup $\Gamma_n(k+1)$ is characteristic in $F_n$, the group $\mathrm{Aut}\,F_n$ naturally acts on $N_{n,k}$.
This action induces a homomorphism
\[ \rho_k: \mathrm{Aut}\,F_n \rightarrow \mathrm{Aut}\,N_{n,k}. \]
Let $\mathcal{A}_n(k)$ be the kernel of $\rho_k$. Then the groups $\mathcal{A}_n(k)$ define a descending central filtration
\[ \mathrm{IA}_n = \mathcal{A}_n(1) \supset \mathcal{A}_n(2) \supset \cdots \]
This filtration is called the Johnson filtration of $\mathrm{Aut}\,F_n$.
Set $\mathrm{gr}^k (\mathcal{A}_n) := \mathcal{A}_n(k)/\mathcal{A}_n(k+1)$.
For each $k \geq 1$, the group $\mathrm{Aut}\,F_n$ acts on $\mathrm{gr}^k (\mathcal{A}_n)$ by conjugation.
This action induces that of $\mathrm{GL}(n,\Z)=\mathrm{Aut}\,F_n/\mathrm{IA}_n$ on it.

\vspace{0.5em}

In order to study the $\mathrm{GL}(n,\Z)$-module structure of ${\mathrm{gr}}^k (\mathcal{A}_n)$,
the Johnson homomorphisms of $\mathrm{Aut}\,F_n$ are defined as follows.
For each $k \geq 1$, define a homomorphism
$\tilde{\tau}_k : \mathcal{A}_n(k) \rightarrow \mathrm{Hom}_{\Z}(H, {\mathcal{L}}_n(k+1))$ by
\[ \sigma \hspace{0.3em} \mapsto \hspace{0.3em} (x \mapsto x^{-1} x^{\sigma}), \hspace{1em} x \in H. \]
Then the kernel of $\tilde{\tau}_k$ is just $\mathcal{A}_n(k+1)$. 
Hence it induces an injective homomorphism
\[ \tau_k : \mathrm{gr}^k (\mathcal{A}_n) \hookrightarrow \mathrm{Hom}_{\Z}(H, \mathcal{L}_n(k+1))
       = H^* \otimes_{\Z} \mathcal{L}_n(k+1). \]
The homomorphism $\tau_k$ is $\mathrm{GL}(n,\Z)$-equivariant, and is called the $k$-th Johnson homomorphism of $\mathrm{Aut}\,F_n$.
Furthermore, we remark that the sum of the Johnson homomorphisms forms a Lie algebra homomorphism as follows.
Let ${\mathrm{gr}}(\mathcal{A}_n) := \bigoplus_{k \geq 1} {\mathrm{gr}}^k (\mathcal{A}_n)$ be the graded sum of
$\mathrm{gr}^k (\mathcal{A}_n)$.
The graded sum ${\mathrm{gr}}(\mathcal{A}_n)$ has a graded Lie algebra structure induced from the commutator bracket on $\mathrm{IA}_n$.
Then the sum of the Johnson homomorphisms
\[ \tau := \bigoplus_{k \geq 1} \tau_k : {\mathrm{gr}}(\mathcal{A}_n) \rightarrow \mathrm{Der}^{+}(\mathcal{L}_n) \]
is a graded Lie algebra homomorphism.

\vspace{0.5em}

It is known that $\tau_1$ gives the abelianization of $\mathrm{IA}_n$ by an independent work of
Cohen-Pakianathan \cite{Co1, Co2}，Farb \cite{Far} and Kawazumi \cite{Kaw} as mentioned above.
Namely, $\mathrm{gr}^1 (\mathcal{A}_n) \cong \mathrm{IA}_n^{\mathrm{ab}}$. Furthermore, we have exact sequences
\begin{eqnarray*}
 & & \hspace{2.5em} 0 \rightarrow \mathrm{gr}^2 (\mathcal{A}_n) \xrightarrow{\tau_2} \mathrm{Hom}_{\Z}(H, \mathcal{L}_n(3))
     \xrightarrow{\mathrm{Tr}_{[2]}} S^2 H \rightarrow 0, \\
 & &  0 \rightarrow \mathrm{gr}_{\Q}^3 (\mathcal{A}_n) \xrightarrow{\tau_{3,\Q}} \mathrm{Hom}_{\Q}(H_{\Q}, \mathcal{L}_{n,\Q}(3))
    \xrightarrow{\mathrm{Tr}_{[3]}^{\Q} \oplus \mathrm{Tr}_{[1^3]}^{\Q}} S^3 H_{\Q} \oplus \Lambda^3 H_{\Q} \rightarrow 0.
\end{eqnarray*}
as $\mathrm{GL}(n,\Z)$-modules. (See \cite{S03} for details.) In general, however, the $\mathrm{GL}(n,\Z)$-module structure of
$\mathrm{gr}_{\Q}^k (\mathcal{A}_n)$ is not determined for $k \geq 4$.

\vspace{0.5em}

To give a lower bound on the image of the Johnson homomorphisms $\tau_k$, or equivalently an upper bound on the cokernel of $\tau_k$,
it is sometimes useful to consider the restriction of $\tilde{\tau}_k$ to the lower central series of $\mathrm{IA}_n$.
Let $\mathcal{A}_n'(k)$ be the lower central series of $\mathrm{IA}_n$ with $\mathcal{A}_n'(1)=\mathrm{IA}_n$.
Since the Johnson filtration is central, $\mathcal{A}_n'(k) \subset \mathcal{A}_n(k)$ for each $k \geq 1$.
Set $\mathrm{gr}^k(\mathcal{A}_n') := \mathcal{A}_n'(k)/\mathcal{A}_n'(k+1)$.
Then $\mathrm{GL}(n,\Z)$ naturally acts on each of $\mathrm{gr}^k(\mathcal{A}_n')$, and
the restriction of $\tilde{\tau}_k$ to $\mathcal{A}_n'(k)$ induces a $\mathrm{GL}(n,\Z)$-equivariant homomorphism
\[ \tau_k' : \mathrm{gr}^k (\mathcal{A}_n') \rightarrow H^* \otimes_{\Z} \mathcal{L}_n(k+1). \]
We also call $\tau_k'$ the Johnson homomorphism of $\mathrm{Aut}\,F_n$.
We remark that if we denote by $i_k : \mathrm{gr}^k(\mathcal{A}_n') \rightarrow \mathrm{gr}^k(\mathcal{A}_n)$ the homomorphism induced from the
inclusion $\mathcal{A}_n'(k) \hookrightarrow \mathcal{A}_n(k)$, then $\tau_k' = \tau_k \circ i_k$ for each $k \geq 1$.

\vspace{0.5em}

Let $\mathcal{C}_n(k)$ be a quotient module of $H^{\otimes k}$ by the action of cyclic group $\Cyc_k$ of order $k$ on the components: 
\[ \mathcal{C}_n(k) = H^{\otimes k} \big{/} \langle a_1 \otimes a_2 \otimes \cdots \otimes a_k - a_2 \otimes a_3 \otimes \cdots \otimes a_k \otimes a_1
   \,|\, a_i \in H \rangle. \]
In our paper \cite{S11}, we showed that for any $k \geq 2$ and $n \geq k+2$,
\begin{equation}\label{eq-dec}
 \mathrm{Coker}(\tau_{k, \Q}') \cong \mathcal{C}_n^{\Q}(k).
\end{equation}

\vspace{0.5em}

Now, by the same argument as $\tau_k$, we can define the Johnson homomorphisms $\tau_k^M$ of $\mathrm{Aut}\,F_n^M$ as follows.
Let $\mathcal{A}_n^M(k)$ be the kernel of a natural homomorphism $\mathrm{Aut}\,F_n^M \rightarrow \mathrm{Aut}\,(F_n/\Gamma_n^M(k+1))$, and
$\mathrm{gr}^k(\mathcal{A}^M):= \mathcal{A}_n^M(k)/\mathcal{A}_n^M(k+1)$ its graded quotient. Then a $\mathrm{GL}(n,\Z)$-equivariant injective
homomorphism
\[ \tau_k^M : \mathrm{gr}^k (\mathcal{A}_n^M) \hookrightarrow H^* \otimes_{\Z} \mathcal{L}_n^M(k+1) \]
is defined by $\sigma \hspace{0.5em} \mapsto \hspace{0.5em} (x \mapsto x^{-1} x^{\sigma})$. In our paper \cite{S06}, we showed that
\[ 0 \rightarrow \mathrm{gr}^k (\mathcal{A}_n^M) \xrightarrow{\tau_k^M} \mathrm{Hom}_{\Z}(H, \mathcal{L}_n^M(k+1))
     \xrightarrow{\mathrm{Tr}_{[k]}^M} S^k H \rightarrow 0 \]
is an exact sequence of $\mathrm{GL}(n,\Z)$-modules for each $k \geq 2$ and $n \geq 4$.

\subsection{Automorphism group of a free nilpotent group}\label{Ss-Nil}
\hspace*{\fill}\ 

In this section we recall some properties of the automorphism group $\mathrm{Aut}\,N_{n,k}$ of a free nilpotent group $N_{n,k}$.
First, we consider generators of $\mathrm{Aut}\,N_{n,k}$.
For any $\sigma \in \mathrm{Aut}\,F_n$, we also denote $\rho_k(\sigma) \in \mathrm{Aut}\,N_{n,k}$ by $\sigma$ if there is no confusion. 
Andreadakis \cite{And} showed that $\rho_2$ is surjective, and that $\rho_k$ is not surjective for $k \geq 3$.
Hence $\mathrm{Aut}\,N_{n,2}$ is generated by the Nielsen's generators $P$, $Q$, $S$ and $U$.

\vspace{0.5em}

For $k \geq 3$, Goryaga \cite{Gor} showed that $\mathrm{Aut}\,N_{n,k}$ is finitely generated for $n \geq 3 \cdot 2^{k-2}+k$ and $k \geq 2$.
In 1984, Andreadakis \cite{An2} showed that $\mathrm{Aut}\,N_{n,k}$ is generated by $P$, $Q$, $S$, $U$ and the other $k-2$ elements for $n \geq k \geq 2$.
In this paper, we use the following Bryant and Gupta's result.
Let $\theta$ be an automorphism of $N_{n,k}$, defined by 
\[ \theta : x_t \mapsto \begin{cases}
               [x_1, [x_2,x_1]]x_1, & t=1, \\
               x_t,                & t \neq 1.
              \end{cases}\]
Then Bryant and Gupta \cite{BrG} showed that for $k \geq 3$ and $n \geq k-1$, the group $\mathrm{Aut}\,N_{n,k}$ is generated by
$P$, $Q$, $S$, $U$ and $\theta$.
In Section {\rmfamily \ref{S-Coh}}, we use these generators to compute the first cohomology group of $\mathrm{Aut}\,N_{n,k}$.
We remark that any presentation for $\mathrm{Aut}\,N_{n,k}$ is not obtained except for $\mathrm{Aut}\,N_{2,k}$ for $k= 1, 2$ and $3$
due to Lin \cite{Lin}.

\vspace{0.5em}

Next, we consider a relation between $\mathrm{Aut}\,N_{n,k}$ and $\mathrm{Aut}\, N_{n,k+1}$.
For $k \geq 1$, we have a natural central group extension
\[ 0 \rightarrow \mathcal{L}_n(k+1) \rightarrow N_{n,k+1} \rightarrow N_{n,k} \rightarrow 1. \]
In this paper, we identify $\mathcal{L}_n(k+1)$ with its image in $N_{n,k+1}$.
Namely, $\mathcal{L}_n(k+1)$ is equal to the $(k+1)$-st term of the lower central series
of $N_{n,k+1}$. In particular, $\mathcal{L}_n(k+1)$ is a characteristic subgroup of $N_{n,k+1}$. Hence the projection $N_{n,k+1} \rightarrow N_{n,k}$
induces a homomorphism $\psi_k : \mathrm{Aut}\, N_{n,k+1} \rightarrow \mathrm{Aut}\,N_{n,k}$.
In order to investigate the kernel of $\psi_k$, we consider the degree $k$ part of the derivation algebra of the free Lie algebra.
For any $f \in \mathrm{Hom}_{\Z}(H, \mathcal{L}_n(k+1))$, define a map $\widetilde{f} : N_{n,k+1} \rightarrow N_{n,k+1}$ by
\[ x^{\widetilde{f}} = ([x])^f x \]
where $[x] \in H$ is the image of $x$ in $H$ under the natural projection $N_{n,k+1} \rightarrow H$.
Then $\widetilde{f}$ is an automorphism of $N_{n,k+1}$, and a map $\mathrm{Hom}_{\Z}(H, \mathcal{L}_n(k+1)) \rightarrow \mathrm{Aut}\, N_{n,k+1}$
defined by $f \mapsto \widetilde{f}$ is an injective homomorphism which image coincides with the kernel of $\psi_k$.
Namely, for $k \geq 1$, we have a group extension
\begin{equation}\label{eq-ext}
 0 \rightarrow \mathrm{Hom}_{\Z}(H, \mathcal{L}_n(k+1)) \rightarrow \mathrm{Aut}\, N_{n,k+1} \rightarrow \mathrm{Aut}\,N_{n,k} \rightarrow 1
\end{equation}
introduced by Andreadakis \cite{And}. (For details, see also Proposition 2.3 in Morita's paper \cite{Mo4}.)

\vspace{0.5em}

Let $T_{n,k}$ be the image of the homomorphism $\rho_k : \mathrm{Aut}\,F_n \rightarrow \mathrm{Aut}\,N_{n,k}$ for each $k \geq 1$.
Clearly, the group $T_{n,k}$ is generated by $P$, $Q$, $S$ and $U$, and is called the tame automorphism group of $N_{n,k}$.
The exact sequence (\ref{eq-ext}) induces
\begin{equation}\label{eq-TAut}
0 \rightarrow \mathrm{gr}^k(\mathcal{A}_n) \rightarrow T_{n,k+1} \rightarrow T_{n,k} \rightarrow 1.
\end{equation}
In Section {\rmfamily \ref{S-Coh}}, we use these two group extensions to study twisted cohomology groups of $T_{n,k}$ and $\mathrm{Aut}\,N_{n,k}$.

\section{Representation theory of $\mathrm{GL}(n,\Q)$ and the symmetric group $\mf{S}_k$}

In this section, we prepare some results in representation theory for $\mathrm{GL}(n,\Q)$, namely Cartan-Weyl's heighest weight theory,
several tensor product theorems and the Schur-Weyl duality for $\mathrm{GL}(n,\Q)$ and the symmetric group $\mf{S}_k$.
In the last of this section, we briefly recall Kra\'{s}kiewicz-Weyman's combinatorial description for the branching rules of irreducible
$\mf{S}_k$-modules to a cyclic subgroup $\Cyc_k$ of order $k$.

\subsection{Partitions and symmetric functions}
\hspace*{\fill}\ 

A partition $\lambda=(\lambda_1, \lambda_2, \ldots)$ is a sequence of decreasing non-negative integers $\lambda_1 \ge \lambda_2 \ge \cdots \ge 0$.
We denote the set of partitions by $\mca{P}$. Set $|\lambda|:=\lambda_1+\lambda_2+ \cdots$. If a partition satisfies $|\lambda|=m$,
the we call $\lambda$ a partition of $m$ and write $\lambda \vdash m$. The conjugate partition of $\lambda$ is the partition
$\lambda'=(\lambda'_1,\lambda'_2, \ldots)$ defined by $\lambda'_i:=\sharp\{j \ | \ \lambda_j \ge i\}$. Put $\ell(\lambda)=\sharp\{i \ | \ \lambda_i\neq 0\}$,
we call it the length of $\lambda$. A partition $\lambda$ is \textit{even} if all $\lambda_i$ are even. For two partitions $\lambda$ and $\mu$,
we write $\lambda \supset \mu$ if $\lambda_i \ge \mu_i$ for all $i$.
For two partition $\lambda$ and $\mu$ satisfying $\lambda \supset \mu$, the skew shape $\lambda \backslash \mu$ is a vertical (resp. horizontal) strip
if there is at most one box in each row (resp. column). 
For a partition $\lambda$ of $m$, a semi-standard (resp. standard) tableaux of shape $\lambda$ is an array $T=(T_{ij})$ of positive integers $1,2, \ldots ,m$ of shape $\lambda$
that is weakly (resp. strictly) increasing in every row and strictly increasing in every column. 
\vspace{0.5em}

For a partition $\lambda=(\lambda_1,\lambda_2, \ldots ,\lambda_n)$, we define a polynomial of $n$-variable by
\[
s_\lambda(x_1,\ldots ,x_n):=\dfrac{\det(x_i^{\lambda_j+n-j})_{1 \le i,j \le n}}{\det(x_i^{n-j})_{1 \le i,j \le n}}.
\]
This is a homogeneous symmetric polynomial of degree $|\lambda|$. We call it the Schur polynomial associated to $\lambda$.
For two partitions $\mu=(\mu_1, \ldots ,\mu_n)$ and $\nu=(\nu_1, \ldots ,\nu_n)$, we define the Littlewood-Richardson coefficients $\LR_{\mu\nu}^{\lambda}$ by 
\[
s_\mu(x_1, \ldots ,x_n) \cdot s_\nu(x_1, \ldots ,x_n)=\sum_{\lambda}\LR_{\mu\nu}^{\lambda}s_\lambda(x_1, \ldots ,x_n).
\]
Then $\LR_{\mu\nu}^\lambda$ becomes a non-negative integer.
\subsection{Highest weight theory for $\mathrm{GL}(n,\Q)$}
\hspace*{\fill}\ 

Let $T_n:=\{\diag(t_1, \ldots ,t_n) \ | \ t_j \neq 0, \ 1 \le j \le n \}$ be the maximal torus of $\mathrm{GL}(n,\Q)$. 
We define one-dimensional representations $\varepsilon_i$ of $T_n$ by $\varepsilon_i(\diag(t_1, \ldots ,t_n))=t_i$. Then 
\begin{eqnarray*}
P_{\mathrm{GL}(n,\Q)}&:=&\{\lambda_1\varepsilon_1+ \cdots +\lambda_n\varepsilon_n \ | \ \lambda_i \in \Z, \,\, 1 \le i \le n \}\cong \Z^n, \\
P^+_{\mathrm{GL}(n,\Q)}&:=&\{\lambda_1\varepsilon_1+ \cdots +\lambda_n\varepsilon_n \in P_{\mathrm{GL}(n,\Q)}\ | \ \lambda_1 \ge \lambda_2 \ge \cdots \ge \lambda_n\}
\end{eqnarray*}
gives the weight lattice and the set of dominant integral weights of $\mathrm{GL}(n,\Q)$ respectively. 
In the following,
for simplicity, we often write $G = \mathrm{GL}(n,\Q)$, $T=T_n$, $P=P_{\mathrm{GL}(n,\Q)}$ and $P^+=P^+_{\mathrm{GL}(n,\Q)}$.
Furthermore, we write $\lambda=(\lambda_1, \ldots, \lambda_n) \in \Z^n$ for $\lambda=\lambda_1\varepsilon_1+ \cdots +\lambda_n\varepsilon_n \in P$ or $P^+$
if there is no confusion.

\vspace{0.5em}

For a rational representation $V$ of $G$, there exists an irreducible decomposition
$V=\bigoplus_{\lambda \in P}V_\lambda$ as a $T$-module where
$V_\lambda:=\{v \in V \ | \ tv=\lambda(t)v \ \text{for any} \ t \in T\}$. We call this decomposition a weight decomposition of $V$ with respect to $T$.
If $V_\lambda \neq \{0\}$, then we call $\lambda$ a weight of $V$. For a weight $\lambda$,
a non-zero vector $v \in V_\lambda$ is call a weight vector of weight $\lambda$.

\vspace{0.5em}

Let $U$ be the subgroup of $G$ consists of all upper unitriangular matrices in $G$. For a rational representation $V$ of $G$,
we set $V^U := \{v \in V \ | \ uv=v \ \text{for all} \ u \in U\}$. We call a non-zero vector $v \in V^U$ a maximal vector of $V$.
This subspace $V^U$ is $T$-stable. Thus, as a $T$-module, $V^U$ has an irreducible decomposition
$V^U=\bigoplus_{\lambda \in P}V^U_\lambda$ where $V^U_\lambda:=V^U \cap V_\lambda$.
\begin{thm}[Cartan-Weyl's highest weight theory] \quad 
\begin{enumerate}[$(1)$]
\item Any rational representation of $V$ is completely reducible.
\item Suppose $V$ is an irreducible rational representation of $G$. Then $V^U$ is one-dimensional, and the weight $\lambda$ of $V^U=V_\lambda^U$ belongs to $P^+$.
We call this $\lambda$ the highest weight of $V$, and any non-zero vector $v \in V^U_\lambda$ is called a highest weight vector of $V$. 
\item For any $\lambda \in P^+$, there exists a unique (up to isomorphism) irreducible rational representation $L^\lambda$ of $G$
with highest weight $\lambda$. Moreover, for two $\lambda, \mu \in P^+$, $L^\lambda \cong L^\mu$ if and only if $\lambda=\mu$.
\item The set of isomorphism classes of irreducible rational representations of $G$ is parameterized by the set $P^+$ of dominant integral weights.
\item Let $V$ be a rational representation of $G$ and $\chi_V$ a character of $V$ as a $T$-module.
Then for two rational representation $V$ and $W$, they are isomorphic as $G$-modules if and only if $\chi_V=\chi_W$.
\end{enumerate}
\end{thm}

\begin{rem}
We can parameterize the set of isomorphism classes of irreducible rational representations of $G$
by $P^+$.
On the other hand, we define the determinant representation by $\det^e:\mathrm{GL}(n,\Q) \ni X \to \det{X}^e \in \Q ^\times$.
The highest weight of this representation is given by $(e,e, \cdots ,e) \in P^+$.
If $\lambda \in P^+$ satisfies $\lambda_n<0$,
then $L^\lambda \cong \det^{-\lambda_n} \otimes L^{(\lambda_1-\lambda_n,\lambda_2-\lambda_n, \ldots ,0)}$.
Therefore we can parameterize the set of isomorphism classes of irreducible rational representations of $G$
by the set $\{(\lambda,e)\}$ where $\lambda$ is a partition such that $\ell(\lambda) \le n$ and $e \in \Z_{<0}$.
Moreover the set of isomorphism classes of polynomial irreducible representations is parameterized by the set of partitions $\lambda$
such that $\ell(\lambda) \le n$.

\vspace{0.5em}

Note that the dual representation of $L^{(\lambda_1,\lambda_2, \ldots ,\lambda_n)}$ is isomorphic to $L^{(-\lambda_n,\ldots ,-\lambda_2,-\lambda_1)}$.
Especially, the natural representation $H_{\Q}=\Q^n$ of $G$ and its dual representation $H_{\Q}^*$ are irreducible with highest weight
$(1,0, \ldots ,0)$ and $(0, \ldots ,0,-1)$ respectively. We also have $H_{\Q}^* \cong \det^{-1} \otimes L^{(1, \ldots ,1,0)}$.

\vspace{0.5em}

There is another parameterization of irreducible rational representations of $G$.
For any $\lambda \in P^+$,
we define two partitions $\lambda_+$ and $\lambda_-$ by
\begin{eqnarray*}
\lambda_+&:=&(\max(\lambda_1,0), \ldots ,\max(\lambda_n,0)), \\
\lambda_-&:=&(-\min(\lambda_n,0), \ldots ,-\min(\lambda_1,0)).
\end{eqnarray*}
Then there is a bijection
\[
 P^+ = \{ (\lambda_1, \lambda_2, \ldots, \lambda_n) \ | \ \lambda_i \in \Z, \,\, 1 \le i \le n \}
   \rightarrow \{(\mu;\nu) \in \mca{P} \times \mca{P} \ | \ \ell(\mu)+\ell(\nu) \le n\}.
\]
defined by $\lambda \mapsto (\lambda_+;\lambda_-)$. Using this bijection, we can parameterize the isomorphism classes of irreducible rational representations
of $G$ by the set $\{ L^{\{\lambda_+;\lambda_-\}} \}$.
Note that under this notation, we have $(L^{\{\lambda;\mu\}})^* \cong L^{\{\mu;\lambda\}}$.
\end{rem}

\begin{thm}[Weyl's character and dimension formula for $\mathrm{GL}(n,\Q)$]\label{thm:dim}\quad
\begin{enumerate}[$(1)$]
\item For a partition $\lambda$, let $\chi_\lambda$ be a character of an irreducible polynomial representation $L^\lambda$.
Then we have $\chi_\lambda(t)=s_\lambda(t_1, \ldots ,t_n)$ for a diagonal matrix $\diag(t_1, \ldots ,t_n) \in T$.
\item For a partition $\lambda$, the dimension of the irreducible polynomial representation $L^\lambda$ coincides with the number of semi-standard tableaux on $\lambda$.
\end{enumerate}
\end{thm}

\subsection{Decompositions of tensor products}
\hspace*{\fill}\ 

In this subsection, we recall some decomposition formulae of tensor products.
\begin{thm}[Pieri's formula]\label{thm:pigl}
Let $\mu$ be a partition such that $\ell(\mu) \le n$.
Then
\[ L^{(1^k)} \otimes L^{\mu} \cong \bigoplus_{\lambda}L^\lambda, \]
where $\lambda$ runs over the set of partitions obtained by adding a vertical
$k$-strip to $\mu$ such that $\ell(\lambda) \le n$.
\end{thm}
\begin{thm}[{\cite[Theorem2.4]{Ko}}]\label{thm:tpgl}
For four partitions $\xi,\eta,\sigma$ and $\tau$ such that $\ell(\xi)+\ell(\eta)+\ell(\sigma)+\ell(\tau) \le n$, we have
\[
L^{\{\xi;\eta\}} \otimes L^{\{\sigma;\tau\}} \cong \bigoplus_{\substack{\lambda, \mu \in \mca{P}, \\[1pt] \ell(\lambda)+\ell(\mu) \le n}}
  \LR_{\{\xi;\eta\},\{\sigma;\tau\}}^{\{\lambda;\mu\}}L^{\{\lambda;\mu\}}.
\]
Here the coefficient $\LR_{\{\xi;\eta\},\{\sigma;\tau\}}^{\{\lambda;\mu\}}$ is defined to be
\[
\sum_{\alpha,\beta,\gamma,\delta \in \mca{P}}\left(\sum_{\kappa \in \mca{P}}\LR_{\kappa\alpha}^{\xi}\LR_{\kappa\beta}^{\tau}\right)\left(\sum_{\pi \in \mca{P}}\LR_{\pi\gamma}^{\eta}\LR_{\pi\delta}^{\sigma}\right)\LR_{\alpha\delta}^{\lambda}\LR_{\beta\gamma}^{\mu}.
\]
\end{thm}

\begin{cor}\label{cor:dtp}
Assume that $1+\ell(\sigma) \le n$. Then we have
\[
L^{\{0;(1)\}} \otimes L^{\{\sigma;0\}} \cong L^{\{\sigma;(1)\}} \oplus \bigoplus_{\lambda,\ell(\lambda) \le n} \LR_{(1),\lambda}^{\sigma}L^{\{\lambda;0\}}.
\]
\end{cor}
\begin{proof}
Under the notation of Theorem \ref{thm:tpgl}, $\xi=\tau=(0)$, $\eta=(1)$ and $\kappa=\alpha=\beta=(0)$.
Thus we obtain
\begin{eqnarray*}
 \LR_{\{0;(1)\},\{\sigma;0\}}^{\{\lambda;\mu\}}
   &=& \sum_{\gamma,\delta}\left(\sum_{\pi}\LR_{\pi\gamma}^{(1)}\LR_{\pi\delta}^{\sigma}\right)\LR_{0\delta}^{\lambda}\LR_{0\gamma}^{\mu}, \\
   &=& \sum_{\pi}\LR_{\pi\mu}^{(1)}\LR_{\pi\lambda}^{\sigma}=\delta_{\mu,(1)}\delta_{\sigma,\lambda}+\LR_{(1),\lambda}^{\sigma}\delta_{0,\mu}.
\end{eqnarray*}
\end{proof}

\begin{cor}[multiplicities of trivial representation]\label{cor:mtv}
If $\ell(\xi)+\ell(\eta)+\ell(\sigma)+\ell(\tau) \le n$, then $[L^{(0)}:L^{\{\xi;\eta\}} \otimes L^{\{\sigma;\tau\}}]=\delta_{\xi, \tau}\delta_{\eta, \sigma}$
where $\delta_{a,b}$ is Kronecker's delta. 
\end{cor}
\begin{proof}
Under the notation of Theorem \ref{thm:tpgl}, $\alpha=\beta=\gamma=\delta=(0)$. Thus $\kappa=\xi=\tau$ and $\pi=\eta=\sigma$.
We obtain $\LR_{\{\xi;\eta\},\{\sigma;\tau\}}^{\{0;0\}}=\delta_{\xi, \tau}\delta_{\eta, \sigma}$.
\end{proof}

\subsection{Schur-Weyl duality}
\hspace*{\fill}\ 

For the natural representation $H_{\Q} \cong L^{(1,0, \ldots, 0)}$ of $\mathrm{GL}(n,\Q)$,
we consider the $m$-th tensor product representation
$\mathrm{GL}(n,\Q) \to \mathrm{GL}(H^{\otimes{m}})$ of $H_{\Q}$.
The symmetric group $\mf{S}_m$ acts on $(H_{\Q})^{\otimes{m}}$ by $\sigma \cdot (v_1 \otimes \cdots \otimes v_m)=v_{\sigma(1)} \otimes \cdots v_{\sigma(m)}$
for $\sigma \in \mf{S}_m$. Since these two actions are commutative, we can decompose $H^{\otimes{m}}$ as
$\mathrm{GL}(n,\Q) \times \mf{S}_m$-module. Let us recall this irreducible decomposition, called the Schur-Weyl duality for $\mathrm{GL}(n,\Q)$ and $\mf{S}_m$.

\begin{thm}[Schur-Weyl duality for $\mathrm{GL}(n,\Q)$ and $\mf{S}_k$]\label{thm:SWgl}\quad 
\begin{enumerate}[$(1)$]
\item Let $\lambda$ be a partition of $m$ such that $\ell(\lambda) \le n$. 
There exists a non-zero maximal vector $v_\lambda$ with weight $\lambda$ satisfying the following three conditions;
\begin{enumerate}
\item[$\mathrm{(i)}$] The $\mf{S}_m$-invariant subspace $S^\lambda:=\sum_{\sigma \in \mf{S}_m}\Q\sigma{v_\lambda}$
                       gives an irreducible representation of $\mf{S}_m$, 
\item[$\mathrm{(ii)}$] The subspace $(H_{\Q}^{\otimes{m}})^U_\lambda$ of weight $\lambda$ coincides with the subspace $S^\lambda$,
\item[$\mathrm{(iii)}$] The $\mathrm{GL}(n,\Q)$-module generated by $v_\lambda$ is isomorphic to the irreducible representation
                      $L^\lambda$ of $\mathrm{GL}(n,\Q)$ associated to $\lambda$.
\end{enumerate}
\item We have the irreducible decomposition:
\[
H_{\Q}^{\otimes{m}} \cong \bigoplus_{\lambda=(\lambda_1 \ge \cdots \ge \lambda_n\ge 0) \vdash m}L^\lambda \boxtimes S^\lambda.
\]
as $\mathrm{GL}(n,\Q) \times \mf{S}_m$-modules.
\item Suppose $n \ge m$. Then $\{S^\lambda \ | \ \lambda \vdash m\}$ gives a complete representatives of irreducible representations of $\mf{S}_m$.
\end{enumerate}
\end{thm}

\subsection{Combinatorial description of branching lows from $\mf{S}_m$ to $\Cyc_m$}
\hspace*{\fill}\ 

Let $\Cyc_m$ be a cyclic group of order $m$. Take a generator $\sigma_m$ of $\Cyc_m$ and a primitive $m$-th root $\zeta_m \in \C$ of unity.
In this section, we consider representations of the cyclic group $\Cyc_m$ over an intermediate field $\Q(\zeta_m) \subset \K \subset \C$.

\vspace{0.5em}

We define one-dimensional representations (or characters) $\chi_m^j:\Cyc_m \to \mb{K}^\times$ by $\chi_m^j(\sigma_m):=\zeta_m^j$ for $0 \le j \le m-1$.
Especially, we denote the trivial representation $\chi_m^0$ by $\triv_m$.
The set of isomorphism classes of irreducible representations of $\Cyc_m$ is given by $\{\chi_m^j, \,\, 0 \le j \le m-1 \}$.

\vspace{0.5em}

Consider $\Cyc_m$ as a subgroup of $\mf{S}_m$ by an embedding $\sigma_m^i \mapsto (1 \, 2 \, \cdots \, m)^i$ for $0 \leq i \leq m-1$.
Let us recall Kra\'{s}kiewicz-Weyman's combinatorial description for the branching rules of irreducible $\mf{S}_m$-modules $S^\lambda$ to the cyclic subgroup $\Cyc_m$.
To do this, first we define a major index of a standard tableau.
%\begin{dfn}\label{def:mj}
For a standard tableau $T$, we define the descent set of $T$ to be the set of entries $i$ in $T$ such that $i+1$ is located in a lower row than that
which $i$ is located. We denote by $D(T)$ the descent set of $T$.
The major index of $T$ is defined by
\[ \maj(T):=\sum_{i \in D(T)}i. \]
If $D(T)=\varphi$, we set $\maj(T)=0$.
%\end{dfn}

\begin{thm}[{\cite{KW}, \cite[Theorem 8.8, 8.9]{Reu}, \cite[Theorem 8.4]{Ga}}]\label{thm:KW}
The multiplicity of $\chi_m^j$ in $\Res_{\Cyc_m}^{\mf{S}_m}S^\lambda$ is equal to the number of standard tableau with shape $\lambda$ satisfying $\maj(T) \equiv j$ modulo $m$.
\end{thm}

\begin{exa}\label{exa:KW} For $m \ge 2$, we have the following table on the multiplicities of $\triv_m=\chi_j^0$ and $\chi_j^1$.
\[
{\small
\begin{array}{c|c|c|c|c} \hline
\lambda & T & \text{major index} & \text{mult. of $\triv_m$} & \text{mult. of $\chi_m^1$} \\
\hline
& & & & \\
(m) & 
\begin{minipage}[h]{30mm}{\hspace{2.5mm}%WinTpicVersion3.08
\unitlength 0.1in
\begin{picture}(  9.7500,  2.4000)(  3.8500, -6.4000)
% BOX 2 0 3 0
% 2 400 400 640 640
% 
\special{pn 8}%
\special{pa 400 400}%
\special{pa 640 400}%
\special{pa 640 640}%
\special{pa 400 640}%
\special{pa 400 400}%
\special{fp}%
% BOX 2 0 3 0
% 2 640 400 880 640
% 
\special{pn 8}%
\special{pa 640 400}%
\special{pa 880 400}%
\special{pa 880 640}%
\special{pa 640 640}%
\special{pa 640 400}%
\special{fp}%
% STR 2 0 3 0
% 3 520 460 520 520 5 0
% $1$
\put(5.2000,-5.2000){\makebox(0,0){$1$}}%
% STR 2 0 3 0
% 3 760 460 760 520 5 0
% $2$
\put(7.6000,-5.2000){\makebox(0,0){$2$}}%
% BOX 2 0 3 0
% 2 880 400 1120 640
% 
\special{pn 8}%
\special{pa 880 400}%
\special{pa 1120 400}%
\special{pa 1120 640}%
\special{pa 880 640}%
\special{pa 880 400}%
\special{fp}%
% BOX 2 0 3 0
% 2 1120 400 1360 640
% 
\special{pn 8}%
\special{pa 1120 400}%
\special{pa 1360 400}%
\special{pa 1360 640}%
\special{pa 1120 640}%
\special{pa 1120 400}%
\special{fp}%
% STR 2 0 3 0
% 3 1240 460 1240 520 5 0
% $m$
\put(12.4000,-5.2000){\makebox(0,0){$m$}}%
% STR 2 0 3 0
% 3 1000 460 1000 520 5 0
% $\cdots$
\put(10.0000,-5.2000){\makebox(0,0){$\cdots$}}%
\end{picture}}
\end{minipage}
& 0 & 1 & 0 \\
& & & & \\
\hline
& & & & \\
(m-1,1) & \begin{minipage}[h]{30mm}{\hspace{2.5mm}%WinTpicVersion3.08
\unitlength 0.1in
\begin{picture}(  9.7500,  4.8000)(  3.8500, -8.8000)
% BOX 2 0 3 0
% 2 400 400 640 640
% 
\special{pn 8}%
\special{pa 400 400}%
\special{pa 640 400}%
\special{pa 640 640}%
\special{pa 400 640}%
\special{pa 400 400}%
\special{fp}%
% STR 2 0 3 0
% 3 520 460 520 520 5 0
% $1$
\put(5.2000,-5.2000){\makebox(0,0){$1$}}%
% BOX 2 0 3 0
% 2 880 400 1120 640
% 
\special{pn 8}%
\special{pa 880 400}%
\special{pa 1120 400}%
\special{pa 1120 640}%
\special{pa 880 640}%
\special{pa 880 400}%
\special{fp}%
% BOX 2 0 3 0
% 2 1120 400 1360 640
% 
\special{pn 8}%
\special{pa 1120 400}%
\special{pa 1360 400}%
\special{pa 1360 640}%
\special{pa 1120 640}%
\special{pa 1120 400}%
\special{fp}%
% STR 2 0 3 0
% 3 1240 460 1240 520 5 0
% $m$
\put(12.4000,-5.2000){\makebox(0,0){$m$}}%
% STR 2 0 3 0
% 3 1000 460 1000 520 5 0
% $\cdots$
\put(10.0000,-5.2000){\makebox(0,0){$\cdots$}}%
% BOX 2 0 3 0
% 2 640 400 880 640
% 
\special{pn 8}%
\special{pa 640 400}%
\special{pa 880 400}%
\special{pa 880 640}%
\special{pa 640 640}%
\special{pa 640 400}%
\special{fp}%
% STR 2 0 3 0
% 3 760 460 760 520 5 0
% $2$
\put(7.6000,-5.2000){\makebox(0,0){$2$}}%
% BOX 2 0 3 0
% 2 400 640 640 880
% 
\special{pn 8}%
\special{pa 400 640}%
\special{pa 640 640}%
\special{pa 640 880}%
\special{pa 400 880}%
\special{pa 400 640}%
\special{fp}%
% STR 2 0 3 0
% 3 520 700 520 760 5 0
% $p$
\put(5.2000,-7.6000){\makebox(0,0){$p$}}%
\end{picture}}
\end{minipage}
%\begin{array}{cccc}
%1 & 2 & \cdots & m \\
%p & & & 
%\end{array}
& p-1 & 0 & 1 \\
& (2 \le p \le m) & & \\
\hline
& & & & \\
(1^m) & \begin{minipage}[h]{30mm}{\hspace{6mm}%WinTpicVersion3.08
\unitlength 0.1in
\begin{picture}(  4.6600,  9.6000)(  1.1000,-13.6000)
% BOX 2 0 3 0
% 2 336 400 576 640
% 
\special{pn 8}%
\special{pa 336 400}%
\special{pa 576 400}%
\special{pa 576 640}%
\special{pa 336 640}%
\special{pa 336 400}%
\special{fp}%
% STR 2 0 3 0
% 3 456 460 456 520 5 0
% $1$
\put(4.5600,-5.2000){\makebox(0,0){$1$}}%
% BOX 2 0 3 0
% 2 336 880 576 1120
% 
\special{pn 8}%
\special{pa 336 880}%
\special{pa 576 880}%
\special{pa 576 1120}%
\special{pa 336 1120}%
\special{pa 336 880}%
\special{fp}%
% BOX 2 0 3 0
% 2 336 1120 576 1360
% 
\special{pn 8}%
\special{pa 336 1120}%
\special{pa 576 1120}%
\special{pa 576 1360}%
\special{pa 336 1360}%
\special{pa 336 1120}%
\special{fp}%
% STR 2 0 3 0
% 3 456 1180 456 1240 5 0
% $m$
\put(4.5600,-12.4000){\makebox(0,0){$m$}}%
% STR 2 0 3 0
% 3 470 910 470 970 5 0
% $\vdots$
\put(4.7000,-9.7000){\makebox(0,0){$\vdots$}}%
% BOX 2 0 3 0
% 2 336 640 576 880
% 
\special{pn 8}%
\special{pa 336 640}%
\special{pa 576 640}%
\special{pa 576 880}%
\special{pa 336 880}%
\special{pa 336 640}%
\special{fp}%
% STR 2 0 3 0
% 3 456 700 456 760 5 0
% $2$
\put(4.5600,-7.6000){\makebox(0,0){$2$}}%
\end{picture}}
\end{minipage}
%\begin{array}{c}
%1 \\
%2 \\
%\vdots \\
%m
%\end{array} 
&
\begin{array}{l}
 \dfrac{m(m-1)}{2} \\
\equiv \left\{
\begin{array}{ll}
0 & \text{if $m$:odd} \\
-\frac{m}{2} & \text{if $m$:even} 
\end{array}
\right.
\end{array}

& \begin{cases} 1 & m : \mathrm{odd} \\ 0 & m : \mathrm{even} \end{cases} & \begin{cases} 1 & m=2 \\ 0 & m \neq 2 \end{cases} \\
& & & & \\
\hline
& & & & \\
(2,1^{m-2}) & \begin{minipage}[h]{30mm}{\hspace{4mm}%WinTpicVersion3.08
\unitlength 0.1in
\begin{picture}(  7.1600,  9.6000)(  1.0000,-13.6000)
% BOX 2 0 3 0
% 2 336 400 576 640
% 
\special{pn 8}%
\special{pa 336 400}%
\special{pa 576 400}%
\special{pa 576 640}%
\special{pa 336 640}%
\special{pa 336 400}%
\special{fp}%
% STR 2 0 3 0
% 3 456 460 456 520 5 0
% $1$
\put(4.5600,-5.2000){\makebox(0,0){$1$}}%
% BOX 2 0 3 0
% 2 336 880 576 1120
% 
\special{pn 8}%
\special{pa 336 880}%
\special{pa 576 880}%
\special{pa 576 1120}%
\special{pa 336 1120}%
\special{pa 336 880}%
\special{fp}%
% BOX 2 0 3 0
% 2 336 1120 576 1360
% 
\special{pn 8}%
\special{pa 336 1120}%
\special{pa 576 1120}%
\special{pa 576 1360}%
\special{pa 336 1360}%
\special{pa 336 1120}%
\special{fp}%
% STR 2 0 3 0
% 3 456 1180 456 1240 5 0
% $m$
\put(4.5600,-12.4000){\makebox(0,0){$m$}}%
% STR 2 0 3 0
% 3 460 920 460 980 5 0
% $\vdots$
\put(4.6000,-9.8000){\makebox(0,0){$\vdots$}}%
% BOX 2 0 3 0
% 2 336 640 576 880
% 
\special{pn 8}%
\special{pa 336 640}%
\special{pa 576 640}%
\special{pa 576 880}%
\special{pa 336 880}%
\special{pa 336 640}%
\special{fp}%
% STR 2 0 3 0
% 3 456 700 456 760 5 0
% $2$
\put(4.5600,-7.6000){\makebox(0,0){$2$}}%
% BOX 2 0 3 0
% 2 576 400 816 640
% 
\special{pn 8}%
\special{pa 576 400}%
\special{pa 816 400}%
\special{pa 816 640}%
\special{pa 576 640}%
\special{pa 576 400}%
\special{fp}%
% STR 2 0 3 0
% 3 696 460 696 520 5 0
% $p$
\put(6.9600,-5.2000){\makebox(0,0){$p$}}%
\end{picture}}
\end{minipage}
%\begin{array}{cc}
%1 & p \\
%2 & \\
%\vdots & \\
%m & 
%&\end{array}
& 
\begin{array}{l}
\dfrac{m(m-1)}{2}-(p-1) \\
\equiv \left\{
\begin{array}{ll}
1-p & \text{if $m$:odd} \\
1-p-\frac{m}{2} & \text{if $m$:even} 
\end{array}
\right.
\end{array}
& \begin{cases} 1 & m : \mathrm{even} \\ 0 & m : \mathrm{odd} \end{cases} & \begin{cases} 1 & m \neq 2 \\ 0 & m = 2 \end{cases} \\
&  (2 \le p \le m)& & \\ \hline
\end{array}
}
\]
\end{exa}

\section{Irreducible decomposition of the derivation algebras over $\Q$}\label{S-HW}

In this section, we give the irreducible decompositions of each of the degree $k$ parts of the derivation algebras
$\mathcal{C}_n^{\Q}$, $\mathrm{Der}^+(\mathcal{L}_{n,\Q})$,
$\mathrm{Der}^+(\mathcal{L}_{n,\Q}^M)$ and $\mathrm{Der}^+(\mathcal{L}_{n,\Q}^N)$.
For describing the multiplicity of $L^\lambda$, we will use some representation of $\Cyc_m$ over an intermediate filed $\Q(\zeta_m) \subset \K \subset \C$.
But we should note that our irreducible decompositions of $\mathcal{C}_n^{\Q}(k)$, $\mathcal{L}_{n,\Q}(k)$,
$\mathrm{Der}^+(\mathcal{L}_{n,\Q})(k)$, $\mathrm{Der}^+(\mathcal{L}_{n,\Q}^M)(k)$ and $\mathrm{Der}^+(\mathcal{L}_{n,\Q}^N)(k)$ hold over $\Q$. 

\vspace{0.5em}

\subsection{Decomposition of $\mathcal{C}_n^{\Q}(k)$}
\hspace*{\fill}\ 

For the natural representation $H_{\Q} \cong L^{(1,0, \ldots, 0)}$ of $\mathrm{GL}(n,\Q)$,
the module $\mathcal{C}^\Q_n(k)$ can be considered as a quotient module of $H_{\Q}^{\otimes{k}}$ by
the action of cyclic subgroup $\Cyc_k$ of $\mf{S}_k$ on the components: 
\[
\mathcal{C}^\Q_n(k) \cong H_{\Q}^{\otimes{k}}
\big{/} \Q\text{-}\span\left\langle v-\sigma_k^iv \,\big{|}\, 1 \le i \le k-1, \hspace{0.5em} \forall{v}\in H_{\Q}^{\otimes{k}} \right\rangle
\]
where $\sigma_k$ is a generator of $\Cyc_k$. In this subsection, we give the irreducible decomposition of $\mathcal{C}^\Q_n(k)$.

\vspace{0.5em}

Since the actions of $\mathrm{GL}(n,\Q)$ and $\mf{S}_k$ are commutative, the space $\mathcal{C}^\Q_n(k)$ is a $\mathrm{GL}(n,\Q)$-module.
Let $\pr:H_{\Q}^{\otimes{k}} \twoheadrightarrow \mathcal{C}_n^\Q(k)$ be the natural projection. The map $\pr$ is a $\mathrm{GL}(n,\Q)$-equivariant homomorphism.

\begin{prop}[irreducible decomposition of $\mathcal{C}^\Q_n(k)$]\label{prop:1}
For any $\lambda=(\lambda_1, \lambda_2, \ldots, \lambda_n) \in P^+$,
the multiplicity of $L^\lambda$ in $\mathcal{C}_n^{\Q}(k)$ as a $\mathrm{GL}(n,\Q)$-module is given by
{\small
\[ [L^\lambda:\mathcal{C}^\Q_n(k)] =  \begin{cases}
                           [\triv_k:\Res_{\Cyc_k}^{\mf{S}_k}S^\lambda]=\dfrac{1}{k} \displaystyle \sum_{\sigma \in \Cyc_k}^{} \chi^\lambda(\sigma), \,\,\,\,
                             & \mathrm{if} \,\,\, \lambda \vdash k, \\
                           0, & \mathrm{if} \,\,\, \mathrm{otherwise}.
                         \end{cases} \]}
Here, $[\triv_k:\Res_{\Cyc_k}^{\mf{S}_k}S^\lambda]$ means the multiplicity of the trivial representation $\triv_k$ of $\Cyc_k$ in the restriction of irreducible $\mf{S}_k$-module $S^\lambda$.
\end{prop}
\begin{proof}
By the Schur-Weyl duality, the complete reducibility and the Schur's lemma, the subspace $\mathcal{C}^\Q_n(k)^U_\lambda$ coincides with $\pr(S^\lambda)$.
Thus $[L^\lambda:\mathcal{C}^\Q_n(k)]=\dim_\Q\pr(S^\lambda)$.
On the other hand, the restriction $\Res_{\Cyc_k}^{\mf{S}_k}S^\lambda$ has a direct isotypic decomposition $S^\lambda\otimes_{\Q}\mathbf{K}=\bigoplus_{j=0}^{k-1}V_j$ over $\mathbf{K}$ 
where $V_j$ is a certain direct sum of the irreducible representation $\chi_j$ of $\Cyc_k$,
namely $V_j=\{v \in S^\lambda \,|\, \sigma_k v=\zeta_k^jv\}$ for $0 \le j \le k-1$.

\vspace{0.5em}

Then, $\Ker(\pr)\otimes_{\Q}\mathbf{K}=\bigoplus_{j \neq 0}V_j$. In fact, if $0 \neq v \in V_j$ for $j \neq 0$, we have $v-\sigma_kv=(1-\zeta_k)v \in \Ker(\pr)$.
Since $1-\zeta_k^j \neq 0$ for $1 \le j \le k-1$, $v \in \Ker(\pr)\otimes_{\Q}\mathbf{K}$. Thus we obtain $V_j \subset \Ker(\pr)\otimes_{\Q}\mathbf{K}$ for $j \neq 0$.
Conversely, for any $v \in H_\K^{\otimes{k}}$, there exists $v_0 \in V_0$ and $v' \in \bigoplus_{j \neq 0}V_j$ such that $v=v_0+v'$.
Since $\bigoplus_{j \neq 0}V_j$ is $\Cyc_k$-stable, $v-\sigma{v}=(v_0+v')-(v_0+\sigma{v'})=v'-\sigma{v'} \in \bigoplus_{j \neq 0}V_j$
for any $\sigma \in \Cyc_k$. Hence $\Ker(\pr)\otimes_{\Q}\mathbf{K} \subset \bigoplus_{j \neq 0}V_j$.

\vspace{0.5em}

Therefore we obtain $\pr(S^\lambda) \cong V_0$ as a vector space, and
$[L^\lambda:\mathcal{C}_n^\Q(k)]=\dim_\Q\pr(S^\lambda)=\dim_\mathbf{K}{V_0}=[\triv_k:\Res_{\Cyc_k}^{\mf{S}_k}S^\lambda]$.
The second equality of the claim follows from an ordinary character theory of finite groups.
\end{proof}

\vspace{0.5em}

For the symmetric product $V=S^k H_{\Q} =L^{(k)}$ or the exterior product $V=\Lambda^k H_{\Q} = L^{(1^k)}$, we calculate the multiplicity
$[V:\mathcal{C}_n^\Q(k)]$ as follows:
\begin{cor}[explicit results as a $\mathrm{GL}(n,\Q)$-module]\label{cor:multc}\quad 
\begin{enumerate}[$(1)$]
\item $[L^{(k)}:\mathcal{C}^\Q_n(k)]=1$. 
\item $[L^{(1^k)}:\mathcal{C}^\Q_n(k)]=\left\{
\begin{array}{ll}
1 & \text{$k$ is odd and $k \le n$,} \\
0 & \text{otherwise}.
\end{array}
\right.$
\item $[L^{(2,1^{k-2})}:\mathcal{C}^\Q_n(k)]=\left\{
\begin{array}{ll}
1 & \text{$k$ is even and $k \le n$,} \\
0 & \text{otherwise}.
\end{array}
\right.$
\end{enumerate}
\end{cor}
\begin{proof}
The claims follows from direct computation for irreducible decompositions of $\Res_{\Cyc_k}^{\mf{S}_k}S^{(k)}$, $\Res_{\Cyc_k}^{\mf{S}_k}S^{(1^k)}$
and $\Res_{\Cyc_k}^{\mf{S}_k}S^{(2,1^{k-2})}$ respectively.
But the claim also follows from Kra\'{s}kiewicz-Weyman's combinatorial description. (See Theorem \ref{thm:KW} and Example \ref{exa:KW}.)
\end{proof}

\subsection{Decomposition of $\mathrm{Der}^+(\mathcal{L}_{n,\Q})$}\label{Ss-Dec}
\hspace*{\fill}\ 

In this subsection,
we consider the irreducible decomposition of the derivation algebra of the free Lie algebra $\mathcal{L}_{n,\Q}$.
The degree $m$ part $\mathcal{L}_{n,\Q}(m)$ of $\mathcal{L}_{n,\Q}$
is a submodule of $H_{\Q}^{\otimes{m}}$ as a representation of $\mathrm{GL}(n,\Q)$.
The irreducible decomposition of $\mathcal{L}_{n,\Q}(m)$ is obtained by the following theorem. 
\begin{thm}[{\cite{Kly}}]\label{thm:Kly}\quad 
\begin{enumerate}[$(1)$]
\item Let us consider a subspace $L_m$ of $H_{\K}^{\otimes{m}}$ generated by all elements $v \in H_{\K}^{\otimes{m}}$ such that $\sigma_mv=\zeta_mv$.
Then $L_m$ becomes a $\mathrm{GL}(n,\K)$-submodule of $H_{\K}^{\otimes{m}}$.
Moreover as $\mathrm{GL}(n,\K)$-module, we have $L_m \cong \mathcal{L}_{n,\Q}(m) \otimes_{\Q} \mathbf{K}$.
\item For $\lambda=(\lambda_1, \lambda_2, \ldots, \lambda_n) \in P^+$,
the multiplicity of $L^\lambda$ in $\mathcal{L}_{n,\Q}(m)$ as a $\mathrm{GL}(n,\Q)$-module is given by
\[ [L^\lambda:\mathcal{L}_{n,\Q}(m)] =
     \begin{cases}
       [\chi_1:\Res_{\Cyc_m}^{\mf{S}_m}S^\lambda], \hspace{2em} & \mathrm{if} \hspace{1em} \lambda_n \ge 0, \,\, (\lambda \,\, \mathrm{is} \,\,
         \mathrm{a} \,\, \mathrm{partition}), \\
       0, & \mathrm{if} \hspace{1em} \mathrm{otherwise}.
     \end{cases} \]
\end{enumerate}
\end{thm}
\begin{rem}
We can obtain a more explicit description of the right hand side by using the M\"{o}bious function as follows:
\begin{eqnarray*}
[\chi_1:\Res_{\Cyc_m}^{\mf{S}_m}S^\lambda]=\dfrac{1}{m}\sum_{g \in \Cyc_m}\overline{\chi_1(g)}\chi^\lambda(g)
  =\dfrac{1}{m}\sum_{d|n}\Mob(d)\chi^{\lambda}(\sigma_m^{n/d}).
\end{eqnarray*}
\end{rem}

\vspace{0.5em}

Next, we consider the irreducible decomposition of $H_{\Q}^* \otimes \mathcal{L}_{n,\Q}(m)$
as a $\mathrm{GL}(n,\Q)$-module.
\begin{prop}[irreducible decomposition of $H_{\Q}^* \otimes \mathcal{L}_{n,\Q}(m)$]\label{prop:2}
For any partition $\lambda=(\lambda_1, \lambda_2, \ldots, \lambda_n)$,
the multiplicity of the irreducible polynomial representation $L^\lambda$ in $H_{\Q}^* \otimes \mathcal{L}_{n,\Q}(m)$ is given by
\[
[L^\lambda:H_{\Q}^* \otimes \mathcal{L}_{n,\Q}(m)]=\sum_{\mu}[L^\mu: \mathcal{L}_{n,\Q}(m)]
\]
where $\mu$ runs over all partitions obtained by removing a vertical $(n-1)$-strip from $(\lambda_1+1, \ldots ,\lambda_n+1)$.
\end{prop}
\begin{proof}
Recall that $H_{\Q}^* \cong \det^{-1} \otimes L^{(1^{n-1},0)}$. Since the highest weight of the irreducible rational representation
$\det^{-1}$ is $(-1, \ldots ,-1)$, we have $\det^{-1} \otimes L^\mu \cong L^{\mu-(1^n)}$.
Thus we obtain
\[ [L^\lambda:H_{\Q}^* \otimes \mathcal{L}_{n,\Q}(m)]=[L^{(\lambda_1+1, \ldots ,\lambda_n+1)}:L^{(1,1, \ldots ,1,0)} \otimes \mathcal{L}_{n,\Q}(m)]. \]
By the Pieri's formula given by Theorem \ref{thm:pigl}, we have $L^{(1^{n-1},0)} \otimes L^{\mu}=\bigoplus_{\nu}{L^\nu}$,
where $\nu$ runs over all partitions obtained by adding a vertical $(n-1)$-strip to $\mu$ such that $\ell(\nu) \le n$. Thus we conclude the claim.
\end{proof}

\vspace{0.5em}

Here we consider the multiplicities of the symmetric product and the exterior product of $H_{\Q}$ in $H_{\Q}^* \otimes \mathcal{L}_{n,\Q}(m)$.
\begin{cor}[explicit results as $\mathrm{GL}(n,\Q)$-modules]\label{cor:2}\quad 
\begin{enumerate}[$(1)$]
\item $[L^{(m-1,0, \ldots ,0)}:H_{\Q}^* \otimes \mathcal{L}_{n,\Q}(m)]=1$.
\item $[L^{(1^{m-1},0)}:H_{\Q}^* \otimes \mathcal{L}_{n,\Q}(m)]=1$ for $1 \le 1+m \le n$.
\end{enumerate}
\end{cor}
\begin{proof}\quad 
\begin{enumerate}[$(1)$]
\item By Proposition \ref{prop:2} and Theorem \ref{thm:Kly}, we have
\begin{eqnarray*}
[L^{(m-1,0, \ldots ,0)}:H_{\Q}^* \otimes \mathcal{L}_{n,\Q}(m)]&=&[L^{(m,1, \ldots ,1)}:L^{(1^{n-1},0)} \otimes \mathcal{L}_{n,\Q}(m)] \\
&=&[L^{(m,0, \ldots ,0)}: \mathcal{L}_{n,\Q}(m)]+[L^{(m-1,1,0, \ldots ,0)}:\mathcal{L}_{n,\Q}(m)] \\
&=&[\chi_m^1:\Res_{\Cyc_m}^{\mf{S}_m}S^{(m)}]+[\chi_m^1:\Res_{\Cyc_m}^{\mf{S}_m}S^{(m-1,1)}].
\end{eqnarray*}
The claim follows from Example \ref{exa:KW}.
\item By Proposition \ref{prop:2} and Theorem \ref{thm:Kly}, we have
\begin{eqnarray*}
[L^{(1^{m-1},0)}:H_{\Q}^* \otimes \mathcal{L}_{n,\Q}(m)]&=&[L^{(2^{m-1},1^{n-m+1})}:L^{(1^{n-1},0)} \otimes \mathcal{L}_{n,\Q}(m)] \\
&=&[L^{(1^m,0)}:\mathcal{L}_{n,\Q}(m)]+[L^{(2,1^{m-2},0)}: \mathcal{L}_{n,\Q}(m)] \\
&=&[\chi_m^1:\Res_{\Cyc_m}^{\mf{S}_m}S^{(1^m)}]+[\chi_m^1:\Res_{\Cyc_m}^{\mf{S}_m}S^{(2,1^{m-2})}].
\end{eqnarray*}
The claim also follows from Example \ref{exa:KW}.
\end{enumerate}
\end{proof}

\vspace{0.5em}

From the corollary above, we verify that each of the multiplicity of $S^k H$ and $\Lambda^k H$ in
$\mathrm{Der}^+(\mathcal{L}_{n,\Q})(k)=H_{\Q}^* \otimes \mathcal{L}_{n,\Q}(k+1)$ is just one.
We can write down each of a maximal vector of $S^k H_{\Q}$ and $\Lambda^k H_{\Q}$ in $\mathrm{Der}^+(\mathcal{L}_{n,\Q})(k)$.
A maximal vector of $S^k H_{\Q}$ is given by
\[ v_{(k)} := \sum_{i=2}^n x_i^* \otimes [x_i, x_1, \ldots, x_1] \in \mathrm{Der}^+(\mathcal{L}_{n,\Q})(k) \]
and that of $\Lambda^k H_{\Q}$ is given by
\[ v_{(1^k)} := \sum_{\sigma \in \mathfrak{S}} \sum_{l \neq \sigma(1)} \mathrm{sgn}(\sigma) x_l^* 
    \otimes [x_l, x_{\sigma(1)}, x_{\sigma(2)}, \ldots, x_{\sigma(k)}] \in \mathrm{Der}^+(\mathcal{L}_{n,\Q})(k).  \]
We leave the proof to the reader as exercises.

\vspace{0.5em}

More generally, we show that the multiplicity of $L^\lambda$ coincides with $[L^\lambda:H_{\Q}^{\otimes{k}}]=\dim{S^\lambda}$.
In other words, the following theorem and corollary describes the kernel of the contraction map $\Phi_{\Q}^k$.
\begin{prop}\label{prop:3}
As a $\mathrm{GL}(n,\Q)$-module, we have a direct decomposition 
\[
H_{\Q}^* \otimes \mathcal{L}_{n,\Q}(k+1) \cong H_{\Q}^{\otimes{k}} \oplus W
\]
for some subrepresentation $W$ such that every irreducible components of $W$ are non-polynomial representations.
Especially, the kernel of the contraction map $\Phi_{\Q}^k$ is isomorphic to $W$.
\end{prop}
\begin{proof}
We shall prove that 
\[
[L^\lambda:H_{\Q}^* \otimes \mathcal{L}_{n,\Q}(k+1)]=\dim{S^\lambda} \,\, (=[L^\lambda: H_{\Q}^{\otimes{k}}])
\]
for any partition $\lambda$ of $k$.
By Proposition \ref{prop:2}, we have 
\[
[L^\lambda:H_{\Q}^* \otimes \mathcal{L}_{n,\Q}(k+1)]=\sum_{\mu}[L^\mu: \mathcal{L}_{n,\Q}(k+1)]
\]
where $\mu$ runs over all elements in the set of partitions obtained by removing a vertical $(n-1)$-strip from $(\lambda_1+1, \ldots ,\lambda_n+1)$.
But for a partition $\lambda$, this set coincides with the set of partitions obtained by adding one box to $\lambda$. Thus, by Theorem \ref{thm:Kly}, we have
\begin{eqnarray*}
[L^\lambda:H_{\Q}^* \otimes \mathcal{L}_{n,\Q}(k+1)]&=&\sum_{x}[L^{\lambda \sqcup\{x\}}:\mathcal{L}_{n,\Q}(k+1)] \\
&=&\sum_{x}[\chi_{1}:\Res_{\Cyc_{k+1}}^{\mf{S}_{k+1}}S^{\lambda \sqcup\{x\}}]
   =\left[\chi_{1}:\bigoplus_{x}\Res_{\Cyc_{k+1}}^{\mf{S}_{k+1}}S^{\lambda \sqcup\{x\}}\right] \\
\end{eqnarray*}
where $x$ runs over all addable boxes to $\lambda$.

\vspace{0.5em}

Recall that $\Ind_{\mf{S}_k}^{\mf{S}_{k+1}}S^\lambda \cong \bigoplus_{x}S^{\lambda \sqcup\{x\}}$
where $x$ runs over the set of addable boxes to $\lambda$. Therefore we obtain
\begin{eqnarray*}
[L^\lambda:H_{\Q}^* \otimes \mathcal{L}_{n,\Q}(k+1)]=
\left[\chi_{1}:\Res_{\Cyc_{k+1}}^{\mf{S}_{k+1}}\Ind_{\mf{S}_k}^{\mf{S}_{k+1}}S^{\lambda}\right]=\left[
S^\lambda:\Res_{\mf{S}_k}^{\mf{S}_{k+1}}\Ind_{\Cyc_{k+1}}^{\mf{S}_{k+1}}\chi_{1}\right].
\end{eqnarray*}
by the Frobenius reciprocity.
We shall use the Mackey's decomposition theorem for $\mf{S}_{k+1}$, $\mf{S}_k$ and $\Cyc_{k+1}$.
But since $\mf{S}_k \cup \Cyc_k$ generates $\mf{S}_{k+1}$, in this case the $(\mf{S}_k,\Cyc_{k+1})$-coset is trivial.
Moreover $\mf{S}_k \cap \Cyc_{k+1}=\{1\}$. Thus by the Mackey's decomposition theorem, we have
\[
\left[S^\lambda:
\Res_{\mf{S}_k}^{\mf{S}_{k+1}}\Ind_{\Cyc_{k+1}}^{\mf{S}_{k+1}}\chi_{\zeta_{k+1}}\right]=\left[S^\lambda:
\Ind_{\{1\}}^{\mf{S}_{k}}\Res^{\Cyc_{k+1}}_{\{1\}}\chi_{\zeta_{k+1}}
\right].
\]
Here $\Ind_{\{1\}}^{\mf{S}_{k}}\Res^{\Cyc_{k+1}}_{\{1\}}\chi_{\zeta_{k+1}}=\Ind_{\{1\}}^{\mf{S}_{k}}(\triv)$ is isomorphic to the regular representation of $\mf{S}_k$.
Since the multiplicity of $S^\lambda$ in the regular representation of $\mf{S}_k$ is equal to the dimension of $S^\lambda$,
we conclude $[L^\lambda:H_{\Q}^* \otimes \mathcal{L}_{n,\Q}(k+1)]=\dim{S^\lambda}$. 
\end{proof}
\begin{cor}\label{cor:3}
Under the notation above, we have
\[ W \cong \bigoplus_{\mu;\ell(\mu) \le n}[L^\mu:\mathcal{L}_{n,\Q}(k+1)]L^{\{\mu;(1)\}}. \]
\end{cor}
\begin{proof}
This follows from Corollary \ref{cor:dtp}.% and Theorem \ref{thm:Kly}.
\end{proof}

\vspace{0.5em}

In \cite{S11}, we have obtained a $\mathrm{GL}(n,\Q)$-equivariant exact sequence
\[ 0 \rightarrow \mathrm{Im}(\tau_{k,\Q}') \rightarrow \mathrm{Der}^+(\mathcal{L}_{n,\Q})(k) \rightarrow \mathcal{C}_n^{\Q}(k) \rightarrow 0 \]
for any $n \geq k+2$. Hence if we fix an integer $k \geq 2$, for any $n \geq k+2$
we can calculate the irreducible decompositions of $\mathcal{C}_n^\Q(k)$ and $\mathrm{Im}(\tau_{k,\Q}')$
using Proposition \ref{prop:1}, Theorem \ref{thm:Kly}, Proposition \ref{prop:2}, Proposition \ref{prop:3} and Corollary \ref{cor:3}.
We give tables of the irreducible decompositions of $\mathcal{C}_n^\Q(k)$ and $\mathrm{Im}(\tau_{k,\Q}')$
in the following.

\vspace{0.5em}

{\small
\begin{center}
{\renewcommand{\arraystretch}{1.3}
\begin{tabular}{|c|l|l|} \hline
  $k$  & \hspace{10em} $\mathcal{C}_n^\Q(k)=\mathrm{Coker}(\tau_{k,\Q}')$, \,\,\, $n \geq k+2$        &                         \\ \hline
  $1$  & $0$                                                 & Andreadakis \cite{And}  \\ \hline
  $2$  & $(2)$                                           & Pettet \cite{Pet}       \\ \hline
  $3$  & $(3) \oplus (1^3)$                          & Satoh \cite{S03}       \\ \hline
  $4$  & $(4) \oplus (2,2) \oplus (2,1^2)$  & Satoh \cite{S09} \\ \hline
  $5$  & $(5) \oplus (3,2) \oplus 2(3,1^2) \oplus (2^2,1) \oplus (1^5)$ & \\ \hline
  $6$  & $(6) \oplus 2(4,2) \oplus 2(4,1^2) \oplus (3^2) \oplus 2(3,2,1) \oplus (3,1^3) \oplus 2(2^3) \oplus (2^2,1^2) \oplus (2,1^4)$ & \\ \hline
  $7$  & $(7) \oplus 2(5,2) \oplus 3(5,1^2) \oplus 2(4,3) \oplus 5(4,2,1) \oplus 2(4,1^3) \oplus 3(3^2,1)$ & \\
& \hspace{1.1em} $\oplus 3(3,2^2)\oplus 5(3,2,1^2) \oplus 3(3,1^4) \oplus 2(2^3,1) \oplus 2(2^2,1^3) \oplus (1^7)$ & \\ \hline
\end{tabular}}
\end{center}
}

In the table above, for simplicity, we write $(\lambda)$ for an irreducible polynomial representation $L^{(\lambda)}$.
For $\mathrm{Im}(\tau_{k,\Q}')$, we have
\begin{eqnarray*}
{\renewcommand{\arraystretch}{1.2}
\begin{array}{|c|c|c|}
\hline
k & \text{polynomial part of} \ \mathrm{Im}(\tau_{k,\Q}') & \text{non-polynomial part of} \ \mathrm{Im}(\tau_{k,\Q}') \\
\hline
1 & (1) & (1,1) \\
\hline
2 & (1^2) & (2,1) \\
\hline
3 & 2(2,1) & (3,1) \oplus (2,1^2) \\
\hline
4 & 3(3,1) \oplus (2^2) \oplus 2(2,1^2) \oplus (1^4) & (4,1) \oplus (3,2) \oplus (3,1^2) \\
& & \oplus (2^2,1) \oplus (2,1^3)\\
\hline
5 & 4(4,1) \oplus 4(3,2) \oplus 4(3,1^2) & (5,1) \oplus (4,2) \oplus 2(4,1^2) \oplus (3^2) \\
 & \oplus 4(2^2,1) \oplus 4(2,1^3)& 3(3,2,1) \oplus (3,1^3) \oplus 2(2^2,1^2) \oplus (2,1^4)\\
\hline
6 & 5(5,1) \oplus 7(4,2) \oplus 8(4,1^2) \oplus 4(3^2) & (6,1) \oplus 2(5,2) \oplus 2(5,1^2) \oplus 2(4,3) \\
& \oplus 14(3,2,1) \oplus 9(3,1^3) \oplus 3(2^3)  & \oplus 5(4,2,1) \oplus 3(4,1^3) \oplus 3(3^2,1) \\
& \oplus 8(2^2,1^2) \oplus 4(2,1^4) \oplus (1^6) & \oplus 3(3,2^2) \oplus 5(3,2,1^2) \oplus 2(3,1^4) \\
& & \oplus 2(2^3,1) \oplus 2(2^2,1^3) \oplus (2,1^5)\\
\hline
7 & 6(6,1) \oplus 12(5,2) \oplus 12(5,1^2) \oplus 12(4,3) & (7,1) \oplus 2(6,2) \oplus 3(6,1^2) \oplus 4(5,3) \\
&  \oplus 30(4,2,1) \oplus 18(4,1^3) \oplus 18(3^2,1)  & \oplus 8(5,2,1) \oplus 4(5,1^3) \oplus (4^2) \oplus 9(4,3,1) \\
& \oplus 18(3,2^3) \oplus 30(3,2,1^2) \oplus 12(3,1^4) & \oplus 6(4,2^2) \oplus 12(4,2,1^2) \oplus 4(4,1^4) \oplus 6(3^2,2) \\
& \oplus 12(2^3,1) \oplus 12(2^2,1^3) \oplus 6(2,1^5) & \oplus 9(3,2^2,1) \oplus 8(3,2,1^3) \oplus 3(3,1^5) \oplus (2^4)\\
& & \oplus 4(2^3,1^2) \oplus 2(2^2,1^4) \oplus (2,1^6) \\
\hline
\end{array}}
\end{eqnarray*}

In the table above, $(\lambda)$ means an irreducible polynomial representation $L^{(\lambda)}$ in the polynomial part,
and $(\mu)$ means an irreducible non-polynomial representation $L^{\{\mu;(1)\}}$ in the non-polynomial part. 

\subsection{Decomposition of $\mathrm{Der}^+(\mathcal{L}_{n,\Q}^M)$}\label{Ss-DecM}
\hspace*{\fill}\ 

We have a basis (\ref{basis-chen}) of $\mathcal{L}_{n}^M(k)$ for each $k \geq 1$.
Note that an element $[x_2,x_1, \ldots ,x_1]$ is a maximal vector with weight $(k-1,1)$ in $\mathcal{L}_{n,\Q}^M(k)$.
The number of elements satisfying $i_1>i_2 \le i_3 \le \cdots \le i_k$ coincides with the number of semi-standard tableau of shape $(k-1,1)$.
Thus we see $\mathcal{L}_{n,\Q}^M(k)$ is isomorphic to the irreducible representation $L^{(k-1,1)}$
as a $\mathrm{GL}(n,\Q)$-module by Theorem \ref{thm:dim}. Using Corollary \ref{cor:dtp}, we have
\begin{prop}\label{prop:Der^M}
For any $k \geq 1$ and $n \geq k+2$,
\[ \mathrm{Der}^+(\mathcal{L}_{n,\Q}^M)(k) =H_{\Q}^* \otimes \mathcal{L}_{n,\Q}^M(k+1)
      \cong L^{\{(k,1);(1)\}} \oplus L^{\{(k);0\}} \oplus L^{\{(k-1,1);0\}}. \]
\end{prop}

Note that $L^{\{(k);0\}}$ is nothing but $S^k H_{\Q}$. In \cite{S06}, we showed that the cokernel of $\tau_k^M$
is isomorphic to $S^k H$ for any $n \geq 4$ and $k \geq 2$.
Hence from Proposition {\rmfamily \ref{prop:Der^M}}, we immediately obtain
\begin{prop}\label{prop:John^M}
For any $k \geq 1$ and $n \geq k+2$,
\[ \mathrm{Im}(\tau_{k,\Q}^M) \cong L^{\{(k,1);(1)\}} \oplus L^{\{(k-1,1);0\}}. \]
\end{prop}

\subsection{Decomposition of $\mathrm{Der}^+(\mathcal{L}_{n,\Q}^N)$}\label{Ss-DecN}
\hspace*{\fill}\ 

We have a basis (\ref{basis-N}) in $\mathcal{L}_{n}^N(k)$ consisting of weight vectors.
Using this basis, we see the character of $\mathcal{L}_{n,\Q}^N(k)$ coincides with that of
$L^{(k-1,1)} \oplus (L^{(k-3,1)} \otimes L^{(1^2)})$. Thus,
\begin{eqnarray*}
   \mathcal{L}_{n,\Q}^N(k) &\cong& L^{(k-1,1)} \oplus (L^{(k-3,1)} \otimes L^{(1^2)}) \\
                           &\cong& L^{(k-1,1)} \oplus L^{(k-2,2)} \oplus L^{(k-2,1^2)} \oplus L^{(k-3,2,1)} \oplus L^{(k-3,1^3)}
\end{eqnarray*}
by Theorem \ref{thm:dim} and Pieri's formula (See Theorem \ref{thm:pigl}.). Using Corollary \ref{cor:dtp}, we have
\begin{prop}\label{prop:Der^N}
Let $W_1$ and $W_2$ be the polynomial part and the non-polynomial part of the irreducible decomposition of
$\mathrm{Der}^+(\mathcal{L}_{n,\Q}^N)(k) = H_{\Q}^* \otimes \mathcal{L}_{n,\Q}^N(k+1)$ respectively. Then we have
\[ W_1 \cong L^{(k)} \oplus 3L^{(k-1,1)} \oplus 2L^{(k-2,2)} \oplus 3L^{(k-2,1^2)} \oplus L^{(k-3,2,1)} \oplus L^{(k-3,1^3)} \]
and
\[ W_2 \cong L^{\{(k,1);(1)\}} \oplus L^{\{(k-1,2);(1)\}}
   \oplus L^{\{(k-1,1^2);(1)\}} \oplus L^{\{(k-2,2,1);(1)\}} \oplus L^{\{(k-2,1^3);(1)\}}. \]
\end{prop}

In our paper \cite{S08}, we investigate the cokernel of the composition map
\[ \tau_{k, N}' : \mathrm{gr}^k(\mathcal{A}_n') \xrightarrow{\tau_k'} H^* \otimes_{\Z} \mathcal{L}_n(k+1)
                  \rightarrow H^* \otimes_{\Z} \mathcal{L}_n(k+1) \]
where the second map is induced from the natural projection $\mathcal{L}_n(k+1) \rightarrow \mathcal{L}_n^N(k+1)$.
In particular, we showed that
\[ \mathrm{Coker}((\tau_{k, N}')_{\Q}) \cong L^{(k)} \oplus L^{(k-2,1^2)} \]
as a $\mathrm{GL}(n,\Q)$-modules. Hence we see that
\begin{prop}\label{prop:John^N}
For any $k \geq 1$ and $n \geq k+2$,
\[ \mathrm{Im}((\tau_{k, N}')_{\Q}) \cong 3L^{(k-1,1)} \oplus 2L^{(k-2,2)} \oplus 2L^{(k-2,1^2)} \oplus L^{(k-3,2,1)} \oplus L^{(k-3,1^3)}
    \oplus W_2. \]
\end{prop}
Here we mention a relation between $\tau_{k, N}'$ and the $k$-th Johnson homomorphism of $\mathrm{Aut}\,F_n^N$.
Let $\mathrm{IA}_n^N$ be the IA-automorphism group of $F_n^N$. Then we can define the Johnson homomorphisms
\[ \eta_{k,N} : \mathrm{gr}^k(\mathcal{A}_n^N) \rightarrow H^* \otimes_{\Z} \mathcal{L}_n^N(k+1) \hspace{0.5em} \mathrm{and} \hspace{0.5em}
   \eta_{k,N}' : \mathrm{gr}^k(\mathcal{A}_n^{'N}) \rightarrow H^* \otimes_{\Z} \mathcal{L}_n^N(k+1) \]
by an argument similar to $\tau_k$ and $\tau_k'$ respectively. (See Section 2.1.2 in \cite{S06} for details.)
Then we have
\[ \mathrm{Im}(\tau_{k, N}') \subset \mathrm{Im}(\eta_{k, N}') \subset \mathrm{Im}(\eta_{k,N}) \subset H^* \otimes_{\Z} \mathcal{L}_n^N(k+1). \]

\section{Abelianization of $\mathrm{Der}^{+}(\mathcal{L}_n^M)$}\label{S-Chen}

In this section, we determine the abelianization of the derivation algebra of the Chen Lie algebra.
To begin with, in order to give an lower bound on it, we consider the Morita's trace map
\[ \mathrm{Tr}_{[k]} : H^* {\otimes}_{\Z} \mathcal{L}_n(k+1) \rightarrow S^k H \]
for $k \geq 2$ as mentioned in Subsection {\rmfamily \ref{Ss-Der}}.
Recently, using these trace maps Morita constructed a surjective graded Lie algebra homomorphism
\[ \Theta := \mathrm{id}_1 \oplus \bigoplus_{k \geq 2} \mathrm{Tr}_{[k]} : \mathrm{Der}^+(\mathcal{L}_n) \rightarrow (H^* \otimes_{\Z} \Lambda^2 H) \oplus
   \bigoplus_{k \geq 2} S^k H \]
where $\mathrm{id}_1$ is the identity map on the degree one part $H^* \otimes_{\Z} \Lambda^2 H$ of $\mathrm{Der}^+(\mathcal{L}_n)$,
and the target is understood to be an abelian Lie algebra.
In particular, Morita showed that $\Theta$ gives the abelianization of $\mathrm{Der}^+(\mathcal{L}_{n})$ up to degree $n(n-1)$,
based on a theorem of Kassabov in \cite{Kas}.
(See Theorem 25 in \cite{Mo3} for details.)

\vspace{0.5em}

On the other hand, $\Theta$ naturally induces a surjective graded Lie algebra homomorphism
\[ \Theta^M := \mathrm{id}_1 \oplus \bigoplus_{k \geq 2} \mathrm{Tr}_{[k]}^M : \mathrm{Der}^+(\mathcal{L}_n^M) \rightarrow (H^* \otimes_{\Z} \Lambda^2 H) \oplus
   \bigoplus_{k \geq 2} S^k H, \]
and hence
\[ \mathrm{Der}^+(\mathcal{L}_n^M)^{\mathrm{ab}} \rightarrow (H^* \otimes_{\Z} \Lambda^2 H) \oplus
   \bigoplus_{k \geq 2} S^k H. \]
In order to prove this is an isomorphism, it suffices to show that for any $k \geq 2$, the degree $k$ part of $\mathrm{Der}^+(\mathcal{L}_n^M)^{\mathrm{ab}}$
is generated by
\[ \binom{n+k-1}{k} = \rank_{\Z} S^k H \]
elements as an abelian group.

\vspace{0.5em}

Let $(\mathrm{Der}^+(\mathcal{L}_n^M))^{\mathrm{ab}}(k)$ be the degree $k$ part of $\mathrm{Der}^+(\mathcal{L}_n^M)^{\mathrm{ab}}$.
Then as a $\Z$-module, $(\mathrm{Der}^+(\mathcal{L}_n^M))^{\mathrm{ab}}(k)$ is generated by
\[ \mathfrak{E} := \{ x_i^* \otimes [x_{i_1},x_{i_2}, \ldots , x_{i_{k+1}}] \,\,|\,\, 1 \leq i, i_j \leq n \} \]
where $x_1^*, \ldots, x_n^*$ is the dual basis of $H^*$ with respect to  $x_1, \ldots, x_n \in H$.
To reduce the generators in $\mathfrak{E}$, we prepare some lemmas. The lemmas below essentially follows from the facts obtained in our previous paper \cite{S06}.
(See \cite{S06} for the proofs.)

\begin{lem}\label{MY1}
Let $l \geq 2$ and $n \geq 2$. For any element
$[x_{i_1},x_{i_2}, x_{j_1}, \ldots , x_{j_l}] \in \mathcal{L}_n^M(l+2)$
and any $\lambda \in \mathfrak{S}_{l}$,
\[ [x_{i_1},x_{i_2}, x_{j_1}, \ldots , x_{j_{l}}]
    = [x_{i_1},x_{i_2}, x_{j_{\lambda(1)}} \ldots , x_{j_{\lambda(l)}}]. \]
\end{lem}

\begin{lem}\label{MY2}
Let $k \geq 2$ and $n \geq 4$. For any $i$ and $i_1, i_2, \ldots , i_{k+1} \in \{1,2 \ldots , n \}$, if $i_1, i_2 \neq i$,
\[ x_i^* \otimes [x_{i_1},x_{i_2}, \ldots , x_{i_{k+1}}] = 0 \in (\mathrm{Der}^+(\mathcal{L}_n^M))^{\mathrm{ab}}(k). \]
\end{lem}

\begin{lem}\label{MY3}
Let $k \geq 2$ and $n \geq 4$. For any $i$ and $i_1, i_2, \ldots , i_{k} \in \{1,2 \ldots , n \}$ such that $i_1, i_2 \neq i$,
and any transposition $\lambda=(m \,\, m+1) \in \mathfrak{S}_k$,
\[ x_{i}^* \otimes [x_{i},x_{i_1}, \ldots , x_{i_{k}}] 
     = x_{i}^* \otimes [x_{i}, x_{i_{\lambda(1)}}, \ldots , x_{i_{\lambda(k)}}] \in (\mathrm{Der}^+(\mathcal{L}_n^M))^{\mathrm{ab}}(k). \]
\end{lem}

\begin{lem}\label{MY4}
Let $k \geq 2$ and $n \geq 4$. For any $i_2, \ldots , i_{k+1} \in \{1,2, \ldots , n \}$, we have
\[ x_{i}^* \otimes [x_{i},x_{i_2}, \ldots , x_{i_{k+1}}]
     = x_{j}^* \otimes [x_{j}, x_{i_2}, \ldots , x_{i_{k+1}}] \in (\mathrm{Der}^+(\mathcal{L}_n^M))^{\mathrm{ab}}(k) \]
for any $i \neq i_2$ and $j \neq i_2, i_{k+1}$.
\end{lem}

Using Lemmas {\rmfamily \ref{MY2}} and {\rmfamily \ref{MY4}}, we see that $(\mathrm{Der}^+(\mathcal{L}_n^M))^{\mathrm{ab}}(k)$ is generated by
\[ \{ x_i^* \otimes [x_{i_1},x_{i_2}, \ldots , x_{i_{k+1}}] \,\,|\,\, 1 \leq i, i_j \leq n, \,\,\, i \neq i_2, i_{k+1} \}. \] 
Furthermore, by Lemma {\rmfamily \ref{MY4}} again,
\[ x_i^* \otimes [x_{i_1},x_{i_2}, \ldots , x_{i_{k+1}}], \hspace{1em} i \neq i_2, i_{k+1} \]
does not depend on the choice of $i$ such that $i \neq i_2, i_{k+1}$. Hence we can set
\[ s(i_1, \ldots, i_k) := x_i^* \otimes [x_{i_1},x_{i_2}, \ldots , x_{i_{k+1}}] \in (\mathrm{Der}^+(\mathcal{L}_n^M))^{\mathrm{ab}}(k) \]
for $i \neq i_1$, $i_k$. On the other hand, take any transposition $\sigma = (m \,\, m+1) \in \mathfrak{S}_k$. If
$2 \leq m \leq k-2$, we see $s(i_1, \ldots, i_k)=s(i_{\sigma(1)}, \ldots, i_{\sigma(k)})$ by Lemma {\rmfamily \ref{MY1}}.
If $m=k-1$, there exists some $1 \leq j \leq n$ such that $j \neq i_1, i_{k-1}, i_k$ since $n \geq 4$. Then we have
\begin{eqnarray*}
   s(i_1, \ldots, i_k) &=& x_j^* \otimes [x_{j},x_{i_1}, \ldots , x_{i_{k}}] = x_j^* \otimes [x_{j},x_{i_1}, \ldots , x_{i_{k-1}}, x_{i_k}] \\
                       &=& s(i_1, \ldots, i_k, i_{k-1}) = s(i_{\sigma(1)}, \ldots, i_{\sigma(k)})
\end{eqnarray*}
by Lemma {\rmfamily \ref{MY3}}. Similarly, we verify that $s(i_1, \ldots, i_k)=s(i_{\sigma(1)}, \ldots, i_{\sigma(k)})$ if $m=1$.
Therefore we conclude that for $n \geq 4$ and $k \geq 2$, $(\mathrm{Der}^+(\mathcal{L}_n^M))^{\mathrm{ab}}(k)$ is generated by
\[ \{ s(i_1, \ldots, i_k) \,\,|\,\, 1 \leq i_1 \leq \cdots \leq i_k \leq n \}. \]
Namely, we obtain
\begin{thm}\label{T-ES_Chen}
For $n \geq 4$, we have
\[ (\mathrm{Der}^+(\mathcal{L}_n^M))^{\mathrm{ab}} \cong (H^* \otimes_{\Z} \Lambda^2 H ) \oplus \bigoplus_{k \geq 2} S^k H. \]
More precisely, this isomorphism is given by the degree one part and the Morita's trace maps $\mathrm{Tr}_{[k]}$.
\end{thm}

This theorem induces a Lie algebra exact sequence
\[ 0 \rightarrow \bigoplus_{k \geq 2} \mathrm{gr}^k(\mathcal{A}_n^M) \xrightarrow{\oplus_{k \geq 2} \tau_k^M}
      \mathrm{Der}^+(\mathcal{L}_n^M) \xrightarrow{\Theta^M} (H^* \otimes_{\Z} \Lambda^2 H ) \oplus \bigoplus_{k \geq 2} S^k H \rightarrow 0. \]

\section{Twisted cohomology groups with coefficients in $\Lambda^l H_{\Q}$}\label{S-Coh}

In general, for any $\mathrm{GL}(n,\Z)$-module $M$, we can naturally regard $M$ as an $\mathrm{Aut}\,F_n$-module and an $\mathrm{Aut}\,N_{n,k}$-module
through the surjective homomorphism
$\mathrm{Aut}\,F_n \rightarrow \mathrm{GL}(n,\Z)$ and $\mathrm{Aut}\,N_{n,k} \rightarrow \mathrm{GL}(n,\Z)$ respectively.
Here we consider the case where $M= \Lambda^l H_{\Q}$ for $l \geq 1$.
In this section, we study the twisted first and second cohomology groups of $T_{n,k}$ and $\mathrm{Aut}\,N_{n,k}$
with coefficients in $\Lambda^l H_{\Q}$.
In particular, we show that the trace map $\mathrm{Tr}_{[1^k]}^{\Q}$ for $\Lambda^k H_{\Q}$ defines a non-trivial cohomology class in
$H^2(T_{n,k}, \Lambda^k H_{\Q})$ for even $k$ and $2 \leq k \leq n$,
and $H^2(\mathrm{Aut}\,N_{n,k}, \Lambda^k H_{\Q})$ for any $3 \leq k \leq n$.

\vspace{0.5em}

In the following, for a group $G$ and a $G$-module $M$, we write $Z^1(G,M)$ for the abelian group of crossed homomorphisms from
$G$ to $M$. We denote by $\delta$ the coboundary operator in group cohomology theory.
We also remark that for any finite group $G$ and $G$-module $M$,
\[ H^p(G, M \otimes_{\Z} \Q) = 0, \hspace{1em} p \geq 1. \]

\vspace{0.5em}

\subsection{Twisted first cohomologies of $\mathrm{GL}(n,\Z)$ and $\mathrm{Aut}\,F_n$}\label{Ss-GLA}
\hspace*{\fill}\ 

Here we show that for $n \geq 3$ the first cohomologies of $\mathrm{GL}(n,\Z)$ and $\mathrm{Aut}\,F_n$ with coefficients in $\Lambda^l H_{\Q}$ are trivial
except for $H^1(\mathrm{Aut}\,F_n, H_{\Q})=\Q$.
\begin{prop}\label{P-Gen}
For $n \geq 3$ and $l \geq 1$, $H^1(\mathrm{GL}(n,\Z), \Lambda^l H_{\Q})=0$.
\end{prop}
\begin{proof}
Take any crossed homomorphism $f: \mathrm{GL}(n,\Z) \rightarrow \Lambda^l H_{\Q}$ and consider the restriction of
$f_{\overline{\Omega}_n}$ to the subgroup $\overline{\Omega}_n$ of $\mathrm{GL}(n,\Z)$.
Since $\overline{\Omega}_n$ is a finite group, $H^1(\overline{\Omega}_n, \Lambda^l H_{\Q})=0$. Hence there exists some $x \in \Lambda^l H_{\Q}$
such that $f_{\overline{\Omega}_n} = \delta x$. Consider a crossed homomorphism
\[ f' := f - \delta x : \mathrm{GL}(n,\Z) \rightarrow \Lambda^l H_{\Q}. \]
Then we see that $f'(\sigma)=0$ for $\sigma=P$, $Q$ and $S$.
Hence it suffices to show $f'(U)=0$. 

\vspace{0.5em}

Set
\[ f'(U) := \sum_{1 \leq i_1 < \cdots < i_l \leq n} a_{i_1, \ldots, i_l} e_{i_1} \wedge \cdots \wedge e_{i_l} \in \Lambda^l H_{\Q} \]
for $a_{i_1, \ldots, i_l} \in \Q$.
Consider a relation $U \, Q^{-(j-1)} S Q^{j-1} = Q^{-(j-1)} S Q^{j-1} \, U$ in $\mathrm{GL}(n,\Z)$.
Since $f'$ is a crossed homomorphism, $f'$ satisfies
\begin{equation}\label{rel-1}
 (Q^{-(j-1)} S Q^{j-1}-1)f'(U)=(U-1)f'(Q^{-(j-1)} S Q^{j-1}) = 0
\end{equation}
since $f'(Q)=f'(S)=0$. 

\vspace{0.5em}

{\bf Case 1.} If $l \geq 3$, for any $a_{i_1, \ldots, i_l}$, we see $3 \leq i_l \leq n$. By observing the coefficients of
$e_{i_1} \wedge \cdots \wedge e_{i_l}$ in (\ref{rel-1})
for $j=i_l$, we obtain $2 a_{i_1, \ldots, i_l}=0$, and hence $a_{i_1, \ldots, i_l}=0$.

\vspace{0.5em}

{\bf Case 2.} Assume $l=2$. We can see $a_{i_1, i_2}=0$ for any $1 \leq i_1 < i_2 \leq n$, except for $(i_1,i_2)=(1,2)$, by the same argument as above.
To show $a_{1,2}=0$, consider the relation {\bf{(R11)}}: $PUPSU =USPS$. Since $f'$ is a crossed homomorphism, $f'$ satisfies
\begin{eqnarray*}
  & & f'(P) + Pf'(U) +PUf'(P) + PUPf'(S) + PUPSf'(U) \\
  & & \hspace{5em} = f'(U) + Uf'(S)+ USf'(P)+ USPf'(S),
\end{eqnarray*}
and hence
\begin{equation}\label{rel-2}
 Pf'(U) +PUPS f'(U) =f'(U)
\end{equation}
by $f'(P)=f'(S)=0$. Observing the coefficients of $e_1 \wedge e_2$, we have $3 a_{1,2}=0$, and hence $a_{1,2}=0$.

\vspace{0.5em}

{\bf Case 3.} Finally, assume $l=1$. Since $a_i=0$ for $3 \leq i \leq n$ by the same argument as above,
it suffices to show $a_1=a_2=0$. By observing the coefficients of $e_2$ in (\ref{rel-2}),
we see $a_1 =0$. Similarly, from the relation {\bf{(R13)}}: $(SU)^2=1$, we have
\[ f'(S)+ S f'(U) + SU f'(S) + SUS f'(U) = 0, \]
and
\[ f'(U) + US f'(U)=0 \]
by $f'(S)=0$. Observing the coefficients of $e_2$ in this equation, we obtain $a_2=0$.

\vspace{0.5em}

This shows that $f' \equiv 0$ as a crossed homomorphism for any case. Namely, $f=\delta x$. Hence we obtain the required results.
This completes the proof of Proposition {\rmfamily \ref{P-Gen}}. 
\end{proof}

\vspace{0.5em}

By the same argument as Proposition {\rmfamily \ref{P-Gen}}, we have
\begin{prop}\label{P-HAut}
For $n \geq 3$ and $l \geq 2$, $H^1(\mathrm{Aut}\,F_n, \Lambda^l H_{\Q})=0$.
\end{prop}

\vspace{0.5em}

Here we should remark that $H^1(\mathrm{Aut}\,F_n,H)=\Z$ for any $n \geq 2$. (See our paper \cite{S01}.) 

\subsection{Twisted cohomologies of $T_{n,k}$}\label{Ss-TAut}
\hspace*{\fill}\ 

Here we consider twisted cohomology groups of the tame automorphism group $T_{n,k}$ of $N_{n,k}$ for $k \geq 2$.
To begin with, from Proposition {\rmfamily \ref{P-HAut}}, we have 
\begin{lem}\label{L-TAut}
For any $n \geq 3$, $k \geq 2$ and $l \geq 2$, $H^1(T_{n,k}, \Lambda^l H_{\Q})=0$.
\end{lem}
\begin{proof}
It is clear from the fact that the induced homomorphism
\[ H^1(\mathrm{Aut}\,N_{n,k}, \Lambda^l H_{\Q}) \rightarrow H^1(\mathrm{Aut}\,F_n, \Lambda^l H_{\Q})=0 \]
from the natural projection $\mathrm{Aut}\,F_n \rightarrow T_{n,k}$ is injective.
\end{proof}

Similarly, for $l=1$, we have an injective homomorphism
\begin{equation}\label{eq-HAN1}
 H^1(T_{n,k}, H) \rightarrow H^1(\mathrm{Aut}\,F_n, H)= \Z
\end{equation}
for $n \geq 2$. Hence, $H^1(T_{n,k}, H)=0$ or $\Z$. In order to show $H^1(T_{n,k}, H)=\Z$, we consider
Morita's crossed homomorphism. Let
\[ \frac{\partial}{\partial x_j} : \Z[F_n] \longrightarrow \Z[F_n] \]
be the Fox's free derivations for $1 \leq j \leq n$. (For a basic material concerning with the Fox's derivation, see
\cite{Bir} for example.)
Let $\frak{a} : \Z[F_n] \rightarrow \Z[H]$ be the ring homomorphism induced from the abelianization $F_n \rightarrow H$.
For any matrix $A=(a_{ij}) \in \mathrm{GL}(n,\Z[F_n])$, set $A^{\mathfrak{a}} = (a_{ij}^{\mathfrak{a}}) \in \mathrm{GL}(n,\Z[H])$.
Then a map
\[ r_M : \mathrm{Aut}\,F_n \longrightarrow \mathrm{GL}(n,\Z[H]) \]
defined by
\[ \sigma \mapsto \biggl{(} \frac{\partial x_i^{\sigma}}{\partial x_j}
    {\biggl{)}}^{\frak{a}} \]
is called the Magnus representation of $\mathrm{Aut}\,F_n$.
We remark that $r_M$ is not a homomorphism but a crossed homomorphism. Namely, $r_M$ satisfies
\[ r_M(\sigma \tau) = r_M(\sigma)^{\tau_*} \cdot {r_M(\tau)} \]
for any $\sigma$, $\tau \in \mathrm{Aut}\,F_n$ where ${r_M(\sigma)}^{\tau_*}$ 
denotes the matrix obtained from $r_M(\sigma)$ by applying a ring homomorphism ${\tau}_{*} : \Z[H] \rightarrow \Z[H]$
induced from $\tau$ on each entry. (For detail for the Magnus representation, see \cite{Mo1}.)

\vspace{0.5em}

Observing the images of the Nielsen's generators by $\mathrm{det} \circ r_M$, we verify that $\mathrm{Im}(\mathrm{det} \circ r_M)$ is contained in
a multiplicative abelian subgroup $\pm H$ of $\Z[H]$. In order to modify the image of $\mathrm{det} \circ r_M$, we consider the signature of $\mathrm{Aut}\,F_n$.
For any $\sigma \in \mathrm{Aut}\,F_n$, set $\mathrm{sgn}(\sigma) := \mathrm{det}(\rho(\sigma)) \in \{ \pm 1 \}$,
and define a map $f_M : \mathrm{Aut}\,F_n \longrightarrow \Z[H]$
by
\[ \sigma \mapsto \mathrm{sgn}(\sigma) \,\, \mathrm{det}(r_M(\sigma)). \]
Then the map $f_M$ is also crossed homomorphism which image is contained in a multiplicative abelian subgroup $H$ in $\Z[H]$.
In the following, we identify the multiplicative abelian group structure of $H$ with the additive one.
Morita \cite{Mo0} showed that the twisted first cohomology group of a mapping class group of a surface with coefficients in $H$
is the infinite cyclic group generated by $f_M$ restricted to the mapping class group.
In our previous paper \cite{S01}, we showed that $f_M$ is a generator of $H^1(\mathrm{Aut}\,F_n, H)= \Z$ for any $n \geq 2$.
We call $f_M$ Morita's crossed homomorphism.

\vspace{0.5em}

Now, we consider the restriction of $f_M$ to the IA-automorphism group $\mathrm{IA}_n$. It is a group homomorphism which target is an abelian group
$H$. On the other hand, $\mathcal{A}_n(2)$ coincide with the commutator subgroup of $\mathrm{IA}_n$ since $\mathrm{gr}^1(\mathcal{A}_n)$
is the abelianization of $\mathrm{IA}_n$ as mentioned above. Hence we see that $f_M(\mathcal{A}_n(2))=0$.
This shows that the crossed homomorphism $f_M : \mathrm{Aut}\,F_n \rightarrow H$ extends to the quotient group
$\mathrm{Aut}\,F_n/\mathcal{A}_n(k) \cong T_{n,k}$ for any $k \geq 2$. We also call this extended crossed homomorphism Morita's crossed homomorphism.
Hence the homomorphism (\ref{eq-HAN1}) is surjective, namely is an isomorphism. Therefore we have
\begin{prop}\label{P-TAut}
For any $n \geq 2$ and $k \geq 2$, $H^1(T_{n,k}, H)=\Z$ which is generated by the Morita's crossed homomorphism.
\end{prop}

Furthermore, we see
\begin{lem}\label{L-TAut2}
For any $n \geq 2$, and $k \geq 2$, the natural projection $T_{n,k+1} \rightarrow T_{n,k}$ induces an isomorphism
\[ H^1(T_{n,k},H) \cong H^1(T_{n,k+1},H). \]
\end{lem}
\begin{proof}
Clearly, the induced homomorphism $H^1(T_{n,k},H) \rightarrow H^1(T_{n,k+1},H)$ from the natural projection $T_{n,k+1} \rightarrow T_{n,k}$
maps the cohomology class of the Morita's crossed homomorphism in $H^1(T_{n,k},H)$ to that in $H^1(T_{n,k+1},H)$.
Hence we see the required results.
\end{proof}

\vspace{0.5em}

Now, from the cohomological five term exact sequence of the group extension (\ref{eq-TAut}), for $k \geq 2$ and $l \geq 1$, we have
\begin{eqnarray*}
    & & 0 \rightarrow H^1(T_{n,k}, \Lambda^l H_{\Q}) \rightarrow H^1(T_{n,k+1}, \Lambda^l H_{\Q}) \\
    & & \hspace{2em} \rightarrow H^1(\mathrm{gr}^k(\mathcal{A}_n), \Lambda^l H_{\Q})^{\mathrm{GL}(n,\Z)}
        \rightarrow H^2(T_{n,k}, \Lambda^l H_{\Q}) \rightarrow H^2(T_{n,k+1}, \Lambda^l H_{\Q}).
\end{eqnarray*}
From Lemmas {\rmfamily \ref{L-TAut}} and {\rmfamily \ref{L-TAut2}}, we see that
\[ 0 \rightarrow H^1(\mathrm{gr}^k(\mathcal{A}_n), \Lambda^l H_{\Q})^{\mathrm{GL}(n,\Z)}
        \rightarrow H^2(T_{n,k}, \Lambda^l H_{\Q}) \rightarrow H^2(T_{n,k+1}, \Lambda^l H_{\Q}) \]
is exact.

\vspace{0.5em}

On the other hand, we have a $\mathrm{GL}(n,\Z)$-equivariant homomorphism $\mathrm{Tr}_{[1^k]} \circ \tau_{k}$, namely
$\mathrm{Tr}_{[1^k]} \circ \tau_{k} \in H^1(\mathrm{gr}^k(\mathcal{A}_n), \Lambda^l H_{\Q})^{\mathrm{GL}(n,\Z)}$.
In \cite{S03}, we showed that $\mathrm{Tr}_{[1^k]} \circ \tau_{k}$ is surjective for even $k$ and $2 \leq k \leq n$.
In particular, we have
\begin{prop}\label{T-ES_TAut}
For even $k$ and $2 \leq k \leq n$, we see $0 \neq \mathrm{tg}(\mathrm{Tr}_{[1^k]} \circ \tau_{k}) \in  H^2(T_{n,k}, \Lambda^k H_{\Q})$
where $\mathrm{tg}$ is the transgression map.
\end{prop}

\subsection{Twisted cohomologies of $\mathrm{Aut}\,N_{n,k}$}\label{Ss-AutN}
\hspace*{\fill}\ 

In this subsection, for $k \geq 3$, we consider twisted cohomology groups of $\mathrm{Aut}\,N_{n,k}$ with coefficients in $\Lambda^l H_{\Q}$.

\begin{prop}\label{P-HAN1}
For $k \geq 3$, $n \geq k-1$ and $l \geq 3$, $H^1(\mathrm{Aut}\,N_{n,k}, \Lambda^l H_{\Q})=0$.
\end{prop}
\begin{proof}
Take any crossed homomorphism $f: \mathrm{Aut}\,N_{n,k} \rightarrow \Lambda^l H_{\Q}$.
Let $g$ be the image of $f$ under the homomorphism
\[ Z^1(\mathrm{Aut}\,N_{n,k}, \Lambda^l H_{\Q}) \rightarrow Z^1(\mathrm{Aut}\,F_n, \Lambda^l H_{\Q}) \]
induced from the natural homomorphism $\mathrm{Aut}\,F_n \rightarrow \mathrm{Aut}\,N_{n,k}$.
Since $H^1(\mathrm{Aut}\,F_n, \Lambda^l H_{\Q})=0$ from Proposition {\rmfamily \ref{P-HAut}}, there exists some $x \in \Lambda^l H_{\Q}$
such that $g = \delta x \in Z^1(\mathrm{Aut}\,F_n, \Lambda^l H_{\Q})$.

\vspace{0.5em}

Set $f':=f-\delta x \in Z^1(\mathrm{Aut}\,N_{n,k}, \Lambda^l H_{\Q})$. Then we have
\[ f'(\sigma) = f(\sigma) - \delta x = g(\sigma) - \delta x=0 \]
for $\sigma = P$, $Q$, $S$ and $U$. Hence it suffices to show $f'(\theta)=0$. Set
\[ f'(\theta) := \sum_{1 \leq i_1 < \cdots < i_l \leq n} b_{i_1, \ldots, i_l} e_{i_1} \wedge \cdots \wedge e_{i_l} \in \Lambda^l H_{\Q} \]
for $b_{i_1, \ldots, i_l} \in \Q$.
Consider a relation $\theta \, Q^{-(j-1)} S Q^{j-1} = Q^{-(j-1)} S Q^{j-1} \, \theta$ in $\mathrm{Aut}\,N_{n,k}$.
Since $f'$ is a crossed homomorphism, $f'$ satisfies
\begin{equation}\label{rel-3}
 (Q^{-(j-1)} S Q^{j-1}-1)f'(\theta)=(\theta-1)f'(Q^{-(j-1)} S Q^{j-1}) = 0
\end{equation}
since $f'(Q)=f'(S)=0$.

\vspace{0.5em}

Since $l \geq 3$, for any $b_{i_1, \ldots, i_l}$, we see $3 \leq i_l \leq n$. By observing the coefficients of 
$e_{i_1} \wedge \cdots \wedge e_{i_l}$ in (\ref{rel-3})
for $j=i_l$, we obtain $2 b_{i_1, \ldots, i_l}=0$, and hence $b_{i_1, \ldots, i_l}=0$.

\vspace{0.5em}

Therefore we see that $f'\equiv0$, and hence $f=\delta x$. This completes the proof of Proposition {\rmfamily \ref{P-HAN1}}. 
\end{proof}

\vspace{0.5em}

By considering the cohomological five term exact sequence of the group extension (\ref{eq-ext}), we obtain
\begin{eqnarray*}
    & & 0 \rightarrow H^1(\mathrm{Aut}\,N_{n,k}, \Lambda^l H_{\Q}) \rightarrow H^1(\mathrm{Aut}\,N_{n,k+1}, \Lambda^l H_{\Q}) \\
    & & \hspace{1em} \rightarrow H^1(H^* \otimes_{\Z} \mathcal{L}_n(k+1), \Lambda^l H_{\Q})^{\mathrm{GL}(n,\Z)}
        \rightarrow H^2(\mathrm{Aut}\,N_{n,k}, \Lambda^l H_{\Q}) \rightarrow H^2(\mathrm{Aut}\,N_{n,k+1}, \Lambda^l H_{\Q})
\end{eqnarray*}
for $k \geq 3$, $n \geq k-1$ and $l \geq 3$.
From Proposition {\rmfamily \ref{P-HAN1}}, we have an exact sequence
\[ 0 \rightarrow H^1(H^* \otimes_{\Z} \mathcal{L}_n(k+1), \Lambda^k H_{\Q})^{\mathrm{GL}(n,\Z)}
        \rightarrow H^2(\mathrm{Aut}\,N_{n,k}, \Lambda^k H_{\Q}) \rightarrow H^2(\mathrm{Aut}\,N_{n,k+1}, \Lambda^k H_{\Q}) \]
for $k \geq 3$ and $n \geq k-1$. Since the trace map
$\mathrm{Tr}_{[1^k]} \in H^1(H^* \otimes_{\Z} \mathcal{L}_n(k+1), \Lambda^k H_{\Q})^{\mathrm{GL}(n,\Z)}$
is surjective for any $3 \leq k \leq n$, we have
\begin{prop}\label{T-ES_NAut}
For $k \geq 3$ and $n \geq k$, we see $0 \neq \mathrm{tg}(\mathrm{Tr}_{[1^k]}) \in H^2(\mathrm{Aut}\,N_{n,k}, \Lambda^k H_{\Q})$
where $\mathrm{tg}$ is the transgression map.
\end{prop}

\vspace{0.3em}

\begin{rem}\label{rem:mult}
Finally, we remark on the multiplicity of $\mathrm{GL}(n,\Q)$-trivial part in 
$\mathrm{Hom}_{\Z}(H_{\Q}^* \otimes_{\Z} \mathcal{L}_{n,\Q}(k+1), \Lambda^k H_{\Q})$ and $\mathrm{Hom}_{\Z}(\mathrm{gr}_{\Q}^k(\mathcal{A}_n), \Lambda^k H_{\Q})$.

\vspace{0,5em}

First of all, if $n \ge k+2$ and an irreducible representation $L^{\{\lambda;\mu\}}$ appears in the decomposition of
$(H_{\Q}^* \otimes_{\Z} \mathcal{L}_{n,\Q}(k+1))^*$, then $\ell(\lambda)+\ell(\mu) \le k+2$ by Corollary \ref{cor:dtp}.
Note that $\Lambda^k H_{\Q} = L^{(1^k)}$ and the multiplicity of $(\Lambda^k H_{\Q})^* = L^{\{0;(1^k)\}}$ is exactly one in
$(H_{\Q}^* \otimes_{\Z} \mathcal{L}_{n,\Q}(k+1))^*$ by the part (2) of Corollary \ref{cor:2}.

\vspace{0.5em}

Using Corollary {\rmfamily \ref{cor:mtv}}, the multiplicity of the $\mathrm{GL}(n,\Q)$-trivial part $L^{(0)}$
in $(H_{\Q}^* \otimes_{\Z} \mathcal{L}_{n,\Q}(k+1))^* \otimes \Lambda^k H_{\Q}$ is exactly one under the condition $n \geq 2k+2$. 
Thus we obtain 
\[
 \mathrm{Hom}_{\Z}(H_{\Q}^* \otimes_{\Z} \mathcal{L}_{n,\Q}(k+1), \Lambda^k H_{\Q})^{\mathrm{GL}(n,\Q)}
    \cong \Q
\]
for $n \geq 2k+2$. Hence, $\mathrm{Hom}_{\Z}(H_{\Q}^* \otimes_{\Z} \mathcal{L}_{n,\Q}(k+1), \Lambda^k H_{\Q})^{\mathrm{GL}(n,\Q)}$
is generated by $\mathrm{Tr}_{[1^k]}^{\Q}$ for $n \geq 2k+2$.

\vspace{0.5em}

Let us recall that $\mathrm{Im}(\tau_{k,\Q}') \subset \mathrm{Im}(\tau_{k,\Q}) \subset H_{\Q}^* \otimes \mathcal{L}_{n,\Q}(k+1)$ and
$\mathrm{Im}(\tau_{k,\Q}) \cong \mathrm{gr}_{\Q}^k(\mathcal{A}_n)$ as $\mathrm{GL}(n,\Q)$-modules.
Since we have a non-zero element
$\mathrm{Tr}_{[1^k]}^{\Q} \circ \tau_{k,\Q} \in \mathrm{Hom}_{\Z}(\mathrm{gr}_{\Q}^k(\mathcal{A}_n), \Lambda^l H_{\Q})^{\mathrm{GL}(n,\Q)}$
for even $k$ and $2 \leq k \leq n$, we obtain 
\[ 
\mathrm{Hom}_{\Z}(\mathrm{gr}_{\Q}^k(\mathcal{A}_n), \Lambda^k H_{\Q})^{\mathrm{GL}(n,\Q)} = \Q 
\]
for even $k$ and $2k+2 \le n$. Hence $\mathrm{Hom}_{\Z}(\mathrm{gr}^k(\mathcal{A}_n), \Lambda^k H_{\Q})^{\mathrm{GL}(n,\Q)}$ is generated by
$\mathrm{Tr}_{[1^k]}^{\Q} \circ \tau_{k,\Q}$.

\vspace{0.5em}

At the present stage, however, it seems to difficult to determine the precise structures of
$H^1(H^* \otimes_{\Z} \mathcal{L}_n(k+1), \Lambda^k H_{\Q})^{\mathrm{GL}(n,\Z)}$ and
$H^1(\mathrm{gr}^k(\mathcal{A}_n), \Lambda^k H_{\Q})^{\mathrm{GL}(n,\Z)}$ in general.
\end{rem}

\section{Acknowledgements}

The second author would like to thank Professor Nariya Kawazumi for helpful comments
for a simple calculation of twisted first cohomology groups with rational coefficients.

Both authors would like to thank Professor Shigeyuki Morita for heplful comments with respect to his work,
and sincere encouragement for our research.

Both authors are supported by JSPS Research Fellowship for Young Scientists and the Global COE program at Kyoto University. 

%%%%%%%%%%%参考文献%%%%

\end{document}